\documentclass[11pt, leqno, twoside]{article}   
\usepackage[applemac]{inputenc}
\usepackage[T1]{fontenc}
\usepackage{mathtools, bm}
\usepackage{amsthm}
\usepackage{amssymb, bm}
\usepackage{bbm}
\usepackage[mathscr]{euscript}
\usepackage[margin=1in]{geometry}
\usepackage{systeme}
\usepackage[french]{babel}
\usepackage{chngcntr}
\usepackage{titling}
\usepackage{cancel}
\usepackage{stmaryrd}
\usepackage{siunitx}
\usepackage{fancyhdr, ifthen}
\usepackage[normalem]{ulem} 

\addto\captionsfrench{%
}

\newenvironment{poliabstract}[1]
  {\renewcommand{\abstractname}{#1}\begin{abstract}}
  {\end{abstract}}

\newtheoremstyle{theoremnoperiod}
  {\topsep}   
  {\topsep}   
  {\normalfont}  
  {0pt}       
  {\bfseries} 
  {}          
  {5pt plus 1pt minus 1pt} 
  {}          
  
\newtheorem*{theorem*}{Théorème}
\newtheorem{theorem}{Théorème}[section]
\newtheorem{corollary}[theorem]{Corollaire}

\newtheorem{lemma}[theorem]{Lemme}
\newtheorem*{conjecture*}{Conjecture}

\newtheorem{property}[theorem]{Proposition}

\numberwithin{equation}{section}

\theoremstyle{theoremnoperiod}
\newtheorem*{thmnodot*}{Théorème}
\newtheorem{thmnodot}[theorem]{Théorème}
\newtheorem{lemmanodot}[theorem]{Lemme}

\DeclareMathOperator{\modulo}{mod}
\DeclareMathOperator{\sign}{sgn}
\DeclareMathOperator{\N}{\mathbb{N}}
\DeclareMathOperator{\Z}{\mathbb{Z}}
\DeclareMathOperator{\R}{\mathbb{R}}

\DeclareMathOperator{\C}{\mathbb{C}}

\DeclareMathOperator{\sC}{\mathcal{C}}

\def\HH{{\mathscr H}}

\DeclareMathOperator{\e}{\rm e}
\def\d{\,{\rm d}}
\def\1{{\bf 1}}
\DeclareMathOperator{\PP}{\mathbb P}

\DeclareMathOperator{\A}{\mathcal A}

\DeclareMathOperator{\sJ}{\mathscr J}
\DeclareMathOperator{\I}{\mathscr I}

\DeclareMathOperator{\cI}{\mathscr I}
\DeclareMathOperator{\cK}{\mathscr K}
\DeclareMathOperator{\cP}{\mathscr P}
\DeclareMathOperator{\cE}{\mathscr E}
\DeclareMathOperator{\cD}{\mathscr D}

\DeclareMathOperator{\gc}{\mathfrak c}
\DeclareMathOperator{\gr}{\mathfrak r}
\DeclareMathOperator{\gs}{\mathfrak s}
\DeclareMathOperator{\gh}{\mathfrak h}

\def\gR{{\mathfrak R}}
\def\p{{p_{m,\nu}}}

\def\res{\mathop{\hbox{\rm R\'es}}\nolimits}

\renewcommand{\leq}{\leqslant}

\renewcommand{\geq}{\geqslant}

\newcommand{\JJ}{\mathcal{J}_\varepsilon}

\usepackage{color}
\definecolor{vert}{rgb}{0,0.5,0}
\definecolor{violet}{rgb}{0.7,0.1,0.8}

\title{Étude statistique du facteur premier médian, 2 : \goodbreak lois locales}
\author{Jonathan Rotgé\thanks{Adresse e-mail : jonathan.rotge@etu.univ-amu.fr\\ \quad 2020 {\it Mathematics Subject Classification}: 11N25, 11N37.\\ {\it Key words and phrases.} middle prime factor, local laws.}\\ \\ {\it\small  Université d'Aix-Marseille,\, Institut de Mathématiques de Marseille CNRS UMR 7373,}\\ {\it\small 163 Avenue De Luminy, Case 907, 13288 Marseille Cedex 9, FRANCE}}
\date{}

\begin{document}
\maketitle
\thispagestyle{empty}

\selectlanguage{english}
\begin{poliabstract}{Abstract} 
We estimate the local laws of the distribution of the middle prime factor of an integer, defined according to multiplicity or not.
An asymptotic estimate with effective remainder is provided for a wide range of values. In particular this enables to precisely describe the  phase transition occurring in the relevant distribution.
	\end{poliabstract}

\selectlanguage{french}
\begin{poliabstract}{Résumé}
   Nous évaluons les lois locales de la répartition du facteur premier médian d'un entier, défini en tenant compte ou non, de la multiplicité. Une formule asymptotique avec terme d'erreur effectif est fournie pour un large domaine de valeurs. En particulier, cela permet de décrire précisément la transition de phase qui s'y opère.
\end{poliabstract}
\bigskip


\section{Introduction et énoncé du résultat}
\subsection{Description  et historique du problème}\label{dhp}
Depuis les travaux fondateurs de Hardy et Ramanujan \cite{bib:hardyrama}, la répartition des facteurs premiers des entiers constitue, au sein de la théorie probabiliste des nombres, un domaine privilégié d'investigation.\par 
Pour tout entier naturel $n\geq 2$, posons
\[\omega(n):=\sum_{p|n}1,\qquad\Omega(n):=\sum_{p^k||n}k,\]
et désignons par $\{q_j(n)\}_{1\leq j\leq\omega(n)}$ (resp. $\{Q_j(n)\}_{1\leq j\leq\Omega(n)}$) la suite croissante des facteurs premiers de~$n$ comptés sans (resp. avec) multiplicité. Le théorème d'Erd\H{o}s-Kac \cite{bib:erdoskac} stipule que, pour $\nu\in\{\omega,\Omega\}$, la relation asymptotique
\[{1\over x}\Big|\Big\{n\leqslant x:\nu(n)\leqslant \log_2x+t\sqrt{\log_2x }\Big\}\Big|=\Phi(t)+o(1)\quad (x\to\infty),\]
où $\Phi(t):=\int_{-\infty}^t \e^{-v^2/2}\d v/\sqrt{2\pi}$ désigne la fonction de répartition de la loi normale, a lieu uniformément en $t\in\R$, et Rényi et Turán \cite{bib:renyi;turan} ont fourni une évaluation optimale du terme d'erreur. Ici et dans la suite, nous notons $\log_k$ la $k$-ième itérée de la fonction logarithme.\par 
Erd\H os \cite{bib:erdos} a énoncé en 1946 une loi du logarithme itéré impliquant que, si $\xi(x)\to\infty$, l'encadrement
\begin{equation}
\label{eq:repart:L2qj+}
	|\log_2 q_j(n)-j|\leq(1+\varepsilon)\sqrt{2j\log_2 j}\quad\big(\xi(x)<j\leq\omega(n)\big)
\end{equation}
est valide pour tous les entiers $n\leqslant x$ sauf au plus $o(x)$ exceptions: voir \cite[th.\thinspace12]{bib:hall_tenenbaum} pour une démonstration.
\par 
La littérature contient de nombreuses variations sur le thème de la répartition usuelle des $q_j(n)$. Mentionnons, à titre d'illustration, que des développements ultérieurs ont permis de l'utiliser pour façonner un modèle du mouvement brownien : voir notamment Manstavi\v{c}ius \cite{bib:mansta1},\,\cite{bib:mansta2},\,\cite{bib:mansta3}.
\par 
Plusieurs travaux récents portent sur l’étude du facteur premier médian
\[\p(n):=\begin{cases}q_{\lceil\omega(n)/2\rceil}(n)&{\rm si\;} \nu=\omega,\\ Q_{\lceil\Omega(n)/2\rceil}(n)&{\rm si\;} \nu=\Omega.\end{cases}\]\par
Améliorant une estimation de De Koninck, Doyon et Ouellet \cite{bib:doyon}, McNew, Singha Roy et Pollack \cite{bib:pollack} ont ainsi obtenu l'évaluation uniforme en $t\in\R$
\begin{equation}\label{eq:freq:KoDoOu}
	\frac1x\Big|\Big\{n\leq x:\log_2\p(n)-\tfrac12\log_2 x<t \sqrt{\log_2 x}\Big\}\Big|=\Phi(2t)+O\bigg(\frac{\{\log_3 x\}^{3/2}}{\sqrt{\log_2 x}}\bigg)\quad(x\geq 16).
\end{equation}
 Ils montrent également que
\begin{equation}
\label{vmNPS}
\sum_{n\leqslant x}\log p_{m,\Omega}(n)=A_{\Omega}x(\log x)^{1/\varphi}\bigg\{1+O\bigg(\frac{\{\log_3 x\}^{3/2}}{\sqrt{\log_2 x}}\bigg)\bigg\}
\end{equation}
où la constante $A_\Omega\approx0{,}414005$\footnote{Dans \cite{bib:pollack}, les auteurs indiquent l'approximation numérique  $A_\Omega\approx1{,}313314$, qui est erronée.} est précisément définie.
Pour certaines plages de valeurs du paramètre $p$, ils proposent en outre une formule asymptotique pour les lois locales
\begin{equation}
	M_\Omega(x,p):=|\{n\leq x:p_{m,\Omega}(n)=p\}|\quad(3\leq p\leq x).
\end{equation}\par
Dans la série d'articles dont le présent travail constitue le second volet, nous nous intéressons à certains problèmes laissés ouverts.  Dans \cite{bib:papier3}, nous montrons que, pour $\nu\in\{\omega,\Omega\}$, le terme d'erreur optimal de \eqref{eq:freq:KoDoOu} est $\asymp1/\sqrt{\log_2 x}$, et que cette valeur est optimale. Dans \cite{bib:papier1} nous obtenons le développement asymptotique 
	\begin{equation}\label{estlogpm}
		\sum_{n\leqslant x}\log p_{m,\nu}(n)=A_\nu x(\log x)^{1/\varphi}\bigg\{1+\sum_{1\leq j\leq J}\frac{\gc_{\nu,j}}{(\log_2 x)^j}+O\bigg(\frac{1}{\{\log_2 x\}^{J+1}}\bigg)\bigg\}\qquad(x\geq3),
	\end{equation}
valable pour tout $J\in\N$ fixé et où  $A_\nu$ et la suite réelle $\{\gc_{\nu,j}\}_{j\in\N^*}$ sont précisément définies. Comme $\gc_{\nu,1}\neq0$, la troncature de \eqref{estlogpm} au premier ordre fournit une formule asymptotique avec terme d'erreur optimal $\asymp 1/\log_2x$, améliorant ainsi significativement  \eqref{vmNPS}.
\par 
Dans cette seconde étude, nous nous proposons de préciser et de généraliser les résultats concernant les lois locales. Il s'agit donc
d'évaluer, dans l'ensemble du domaine $\log_2p\asymp\log_2x$, les quantités
\begin{equation}\label{def:Mnu}
	M_\nu(x,p):=|\{n\leq x:\p(n)=p\}|\quad(\nu\in\{\omega,\Omega\},\,3\leq p\leq x).
\end{equation}
Nous verrons que le cas $\nu=\Omega$, seul considéré dans \cite{bib:pollack}, est de difficulté technique très supérieure au cas $\nu=\omega$. Ce phénomène remarquable tient à l'influence des petits facteurs sur la taille du facteur premier médian. Leur présence ne perturbe significativement la croissance de la suite des facteurs premiers que s'ils sont comptés avec multiplicité. Ainsi, le comportement de $M_\omega(x,p)$ s'avère régulier dans toute la zone $\log_2p\asymp \log_2x$, mais celui de $M_\Omega(x,p)$ présente un changement de phase autour de la valeur critique $\beta_p:=(\log_2 p)/\log_2 x=\tfrac15$. \par
Alors que le travail \cite{bib:pollack} ne fournit pas d'estimation pour $M_\Omega(x,p)$ lorsque $|\beta_p-\frac15|\leqslant \varepsilon$, nous nous sommes particulièrement attachés à décrire précisément la transition de phase. Hors de  cette zone critique, nous proposons une amélioration significative des termes d'erreurs relatifs obtenus dans~\cite{bib:pollack}. Les résultats relatifs au  cas $\nu=\omega$ sont  par ailleurs nouveaux. 
 
\par
Mentionnons enfin que McNew, Pollack et Singha Roy \cite{bib:pollack} considèrent plus généralement le facteur premier $\alpha$-positionné défini par $p_{\Omega}^{(\alpha)}(n):=Q_{\lceil\alpha\Omega(n)\rceil}\quad (0<\alpha<1)$. Nos méthodes sont adaptables sans difficulté à ce cas. Nous avons préféré restreindre l'étude au cas $\alpha=\tfrac12$ afin de préserver la clarté d'exposition et éviter la prolifération de détails techniques sans intérêt théorique.
\subsection{Notations}
Dans toute la suite, les lettres $p$ et $q$ désignent des nombres premiers. Notons $\kappa$ la constante de Mertens et $\gamma$ la constante d'Euler-Mascheroni. Définissons alors \begin{equation}\label{def:Hnu;H0}
	\begin{aligned}
		\HH_\nu(z)&:=\begin{cases}
			\displaystyle\e^{\kappa}\prod_{q}\Big(1+\frac{z}{q-1}\Big)\e^{-z/q}\quad(z\in\C)\qquad&\textnormal{si }\nu=\omega,\\
			\displaystyle\e^{\gamma z}\prod_{q}\Big(1-\frac{1}{q}\Big)^z\Big(1-\frac{z}{q}\Big)^{-1}\quad(\Re z<2)&\textnormal{si }\nu=\Omega,
		\end{cases}\\
		\HH_\Omega^*(z)&:=\big(1-\tfrac12z\big)\HH_\Omega(z)\quad (\Re z<3),\quad \gh:=\HH_\Omega^*(2)=\tfrac14\e^{2\gamma}\prod_{q\geq 3}\Big(1+\frac{1}{q(q-2)}\Big)\approx 1{,}201304.
	\end{aligned}
\end{equation}
Notons $a_\nu:=0$ si $\nu=\omega$, $a_\nu:=\tfrac15$ si $\nu=\Omega$, définissons
\begin{equation}\label{def:fnu;rho;rhonu}
	\begin{gathered}
		f_\nu(z):=\frac{\HH_\nu(z)\e^{-\gamma/z}}{\Gamma(1+1/z)}\quad (0<\Re z<2),\qquad \gc:=-\tfrac34\res(f_\Omega;2)=\frac{3\gh\e^{-\gamma/2}}{\sqrt\pi},\\
		\varrho_\nu(v):=\frac{(1+w)f_\nu(w)}{2w\sqrt{\pi vw}}\quad \big(a_\nu<v<1,\,w:=\sqrt{(1-v)/v}\big),
	\end{gathered}
\end{equation}
et posons
\begin{equation}\label{def:betap;betapstar;varepsilonx;deltap;taup}
	\beta_p=\beta_p(x):=\frac{\log_2 p}{\log_2 x},\qquad\varepsilon_x:=\frac1{\log_2 x},\qquad\delta_p:=\beta_p-\tfrac15\quad(3\leq p\leq x).
\end{equation}
Notons d'emblée que 
\begin{equation}
	p=\e^{(\log x)^{\beta_p}},\quad-\tfrac15\leqslant \delta_p\leqslant \tfrac45\qquad (3\leqslant p\leqslant x).
\end{equation}
Désignons par $\PP$ l'ensemble des nombres premiers et considérons le domaine 
\begin{equation}\label{def:domaine:D}
	\cD_\varepsilon=\big\{(x,p)\in[3,\infty[\times\PP:\e^{(\log x)^\varepsilon}\leqslant p\leqslant \e^{(\log x)^{1-\varepsilon}}\big\}\quad\big(0<\varepsilon<\tfrac12\big).
\end{equation}\par 
Posons, pour $3\leq p\leq x$,
\begin{equation}\label{def:eq:gR;gammanu;Psi}
	\begin{gathered}
		\gR_1:=\varepsilon_x+\frac{\sqrt{\varepsilon_x}}{|\delta_p|(\log p)^{\delta_p^2/4}},\quad\gR_2:=\sqrt{\varepsilon_x}+\frac{|\delta_p|^3}{\varepsilon_x},\quad \gR_3:=\frac{\varepsilon_x}{\delta_p^2},\\
		\gamma_\nu(v):=\begin{cases}
			\tfrac12(1-3v)\qquad&\textnormal{si }0< v\leq a_\nu,\\
			1-2\sqrt{v(1-v)}&\textnormal{si } a_\nu \leq v<1,
		\end{cases}\qquad \Psi(v):=\frac{\e^{\max(v,0)^2/2}}{\sqrt{2\pi}}\int_{v}^{\infty}\e^{-t^2/2}\d t.
	\end{gathered}
\end{equation}
Nous avons ainsi
\[\Psi(v)=\frac1{\sqrt{2\pi}v}+O\Big(\frac{1}{v^3}\Big)\quad(v\to\infty),\quad \Psi(v)=\tfrac12+O(v)\quad(v\to0).\]
\subsection{Résultat principal}
Notre objectif est de démontrer le résultat suivant.

\begin{theorem}\label{th:eq:Mp}
	Soit $0<\varepsilon<\tfrac12$. Sous la condition $(x,p)\in\cD_\varepsilon$, nous avons uniformément
	\begin{gather}
		\label{eq:Mp:pomega}M_\omega(x,p)=\frac{\{1+O(\varepsilon_x)\}\varrho_\omega(\beta_p)x}{p(\log x)^{\gamma_\omega(\beta_p)}\sqrt{\log_2 x}},\\[10pt]
		\label{eq:Mp:Gomega}M_{\Omega}(x,p)=\begin{cases}
			\displaystyle\frac{\{1+O(\gR_1)\}\gc x}{p(\log x)^{\gamma_\Omega(\beta_p)}}=\frac{\{1+O(\gR_1)\}\gc x(\log p)^{3/2}}{p\sqrt{\log x}}&\textnormal{si }\delta_p\leq-\sqrt{10\varepsilon_x\log_3x},\\[15pt]
			\displaystyle\frac{\{1+O(\gR_2)\}\gc x\Psi\big(\frac14\delta_p\sqrt{125/\varepsilon_x}\big)}{p(\log x)^{\gamma_\Omega(\beta_p)}}&\textnormal{si }-\sqrt{10\varepsilon_x\log_3x}\leq\delta_p\leq\varepsilon_x^{2/5},\\[15pt]
			\displaystyle\frac{\{1+O(\gR_3)\}\varrho_\Omega(\beta_p)x}{p(\log x)^{\gamma_\Omega(\beta_p)}\sqrt{\log_2 x}}\qquad&\textnormal{si }\delta_p\geq\varepsilon_x^{2/5}.
		\end{cases}
	\end{gather}
\end{theorem}


\noindent{\it Remarques.} (i) La fonction $\gamma_\Omega$ est de classe $\sC^1$ sur $]0,1[$.\par\medskip
(ii) Il est à noter que $M_\Omega(x,p)>M_\omega(x,p)$ pour $\beta_p\leqslant \frac15$. Cela reflète la disparité de l'influence des petits facteurs premiers mentionnée au paragraphe \ref{dhp}.\par 
\medskip
(iii) Dans le domaine $\delta_p\leq -\sqrt{10\varepsilon_x\log_3x}$, nous avons $\gR_1\ll\sqrt{\varepsilon_x}$.
\par\medskip
(iv) À la frontière $\delta_p=-\sqrt{10\varepsilon_x\log_3x}$, nous avons
\[\gR_2\asymp\big(\log_3x\big)^{3/2}\sqrt{\varepsilon_x}\asymp\gR_1\big(\log_3x\big)^2.\]\par
(v) Au point critique $\beta_p=\tfrac15$, nous avons
\begin{equation*}
	M_\Omega\big(x,p\big)=\frac{\{1+O(\sqrt{\varepsilon_x})\}\gc x}{2p\log p}\quad \big({p\geq 3,\,x=\e^{(\log p)^5}}\big).
\end{equation*}
\par
(vi) Dès que $\delta_p\gg 1$, nous avons $\gR_1\asymp\gR_3\asymp\varepsilon_x$, améliorant ainsi le résultat principal de \cite{bib:pollack}. 
\par
(vii) Un ingrédient essentiel de la preuve du Théorème \ref{th:eq:Mp} est l'estimation uniforme de la quantité
\[\sum_{\substack{P^+(n)\leqslant y\\ \Omega(n)=k}}\frac1n\quad(y\geq 3,\, k\geq 1),\]
obtenue en \eqref{eq:lambdaOmega:unif} et qui
généralise et précise un résultat récent de Lichtman \cite{lichtman}.\par
\par 
\smallskip
 Nous démontrons la formule \eqref{eq:Mp:Gomega} en trois étapes. Au Théorème \ref{th:eq:Mp:critZone:out}, nous estimons $M_\Omega(x,p)$ pour des valeurs de $\beta_p$ satisfaisant à la condition $|\delta_p|\geq\sqrt{\varepsilon_x}$. Le Théorème \ref{th:eq:Mp:critZone:in} fournit une estimation analogue dans le domaine $|\delta_p|\leq\varepsilon_x^{1/3}$. Enfin, à la section \ref{sec:recollement}, nous optimisons le terme d'erreur dans l'intersection des domaines de validité des formules \eqref{eq:Mp:critZone:out:Gomega} et \eqref{eq:Mp:critZone:in}.



\section{Domaine de contribution principale}
Posons
\[\chi_{\nu,p}(n):=\begin{cases}
	1\quad&\textnormal{si }\p(n)=p,\\
	0&\textnormal{sinon},
\end{cases}\quad(3\leq p\leq n\leq x),\]
de sorte que
\begin{equation}\label{eq:Mnup:defbis}
	M_\nu(x,p)=\sum_{n\leq x}\chi_{\nu,p}(n).
\end{equation}
Ce paragraphe est consacré à mettre en évidence une plage de valeurs de $\nu(n)$ dominant la som\-me~\eqref{eq:Mnup:defbis}. Nous aurons usage des deux lemmes techniques suivants.


\begin{lemmanodot}\label{l:majo:bigOmega}{\bf \cite[lemme 2.1]{bib:papier1}.}
L'estimation
	\begin{equation}\label{eq:majo:bigOmega}
		\sum_{\substack{n\leq x\\\Omega(n)\geq k}}1\ll\frac{kx\log x}{2^k},
	\end{equation}
	a lieu uniformément pour $k\geq 1$, $x\geq 2$.
\end{lemmanodot}


Posons
\begin{equation}\label{def:Q}
	Q(v):=v\log v-v+1\quad(0<v<1).
\end{equation}


\begin{lemmanodot}\label{l:norton}{\bf \cite[lemmes 4.5 et 4.7]{bib:norton}.}
	Soient $0<a<1<b$. Pour $v\geqslant 1$, nous avons
	\begin{align}
		\sum_{n\leq av}\frac{v^n}{n!}&<\frac{\e^{v(1-Q(a))}}{(1-a)\sqrt{av}},\label{eq:majo:norton:inf}\\
		\sum_{n\geq bv}\frac{v^n}{n!}&<\frac{\sqrt b\e^{v(1-Q(b))}}{(b-1)\sqrt{2\pi v}}.\label{eq:majo:norton:sup}
	\end{align}
\end{lemmanodot}


Soit $W_{-1}:[-1/\e,0[\to ]-\infty,-1]$ la branche négative de la fonction de Lambert, réciproque de la fonction $z\mapsto z\e^z$, de sorte que $w(t):=\e^{1+W_{-1}(t)}$ est l'unique solution dans $]0,1[$ de $w(\log w-1)=\e t$ pour $t\in]-1/\e,0[$. Posons alors
\[\kappa_\varepsilon(v):=w\Big(\frac{\varepsilon-2\sqrt{v(1-v)}}{\e}\Big)\quad\big(0<\varepsilon<\tfrac12,\,\varepsilon\leq v \leq1-\varepsilon\big).\]
Ainsi $\kappa_\varepsilon(\beta_p)$ est l'unique solution dans $]0,1[$ de l'équation $v(1-\log v)=2\sqrt{\beta_p(1-\beta_p)}-\varepsilon$.
La fonction $\kappa_\varepsilon$ est concave sur $[\varepsilon,1-\varepsilon]$ et atteint son maximum $w(\{\varepsilon-1\}/\e)<1$ en $v=\tfrac12$.
\par 
Définissons encore, pour $3\leq p\leq x$, $0<\varepsilon<\tfrac12$,
\[\A_{\nu,\varepsilon}(p):=\Big\{n\leqslant x:\kappa_\varepsilon(\beta_p)<\frac{\nu(n)}{\log_2x}<\frac2{\log 2}\Big\}.\] 


\begin{lemma}\label{l:majo:Mp:smallbignu}
	Soit $0<\varepsilon<\tfrac12$. Sous les conditions  $x\geqslant 3$, $\varepsilon\leq\beta_p\leq1-\varepsilon$, nous avons uniformément
	\begin{equation}\label{eq:majo:Mp:smallbignu}
		\sum_{n\in[1,x)\smallsetminus\A_{\nu,\varepsilon}(p)}\chi_{\nu,p}(n)\ll \frac{x}{p(\log x)^{\gamma_\nu(\beta_p)+\varepsilon}}.
	\end{equation}
\end{lemma}

\begin{proof} Une application de \eqref{eq:majo:bigOmega} avec $k:=\lfloor2(\log_2 x)/\log 2\rfloor$ fournit
	\begin{equation}\label{eq:majo:inter:Mp:bignu}
		\sum_{\substack{n\leq x\\ \nu(n)\geq 2(\log_2 x)/\log 2}}\chi_{\nu,p}(n)\leq \sum_{\substack{d\leq x/p\\ \Omega(d)\geq\lfloor 2(\log_2 x)/\log 2\rfloor-1}}1\ll \frac{x\log_2 x}{p\log x}.
	\end{equation}
Cette majoration est compatible avec \eqref{eq:majo:Mp:smallbignu}.	
\par Par ailleurs, d'après l'inégalité de Hardy-Ramanujan nous avons, pour une constante positive $c_0$ convenable,
	\begin{equation}\label{eq:majo:HR}
		\sum_{\substack{n\leq x\\\omega(n)=k}}1\ll\frac{x(\log_2 x+c_0)^{k-1}}{(k-1)!\log x}\quad(x\geq 3,\,k\geqslant 1).
	\end{equation}
	Puisque $\kappa_\varepsilon(\beta_p)\leqslant w((\varepsilon-1)/\e)<1$ pour $0<\varepsilon<\tfrac12$, $\varepsilon\leq\beta_p\leq1-\varepsilon$,  nous obtenons par \eqref{eq:majo:norton:inf}
	\begin{equation}\label{eq:majo:Mp:inter:smallnu:HR}
		\begin{aligned}
			\sum_{\substack{n\leq x\\ \nu(n)\leq \kappa_\varepsilon(\beta_p)\log_2 x}}\chi_{\nu,p}(x)&\leq \sum_{\substack{d\leq x/p\\\omega(d)\leq\kappa_\varepsilon(\beta_p)\log_2 x}}1\ll\frac{x}{p\log x}\sum_{k\leq \kappa_\varepsilon(\beta_p)\log_2 x}\frac{(\log_2 x)^{k-1}}{(k-1)!}\\
			&\ll\frac{x}{p(\log x)^{Q(\kappa_\varepsilon(\beta_p))}},
		\end{aligned}
	\end{equation}
	où $Q$ est la fonction définie en \eqref{def:Q}. Puisque $\kappa_\varepsilon(\beta_p)$ est l'unique solution dans $]0,1[$ de l'équation $v(1-\log v)=2\sqrt{\beta_p(1-\beta_p)}-\varepsilon$, nous obtenons que le membre de gauche de \eqref{eq:majo:Mp:inter:smallnu:HR} est
	\begin{equation}\label{eq:majo:inter:Mp:smallnu}
		\ll\frac{x}{p(\log x)^{1-2\sqrt{\beta_p(1-\beta_p)}+\varepsilon}}.
	\end{equation}
En regroupant les estimations \eqref{eq:majo:inter:Mp:bignu} et \eqref{eq:majo:inter:Mp:smallnu}, nous obtenons l'estimation annoncée puisque $\gamma_\nu(v)\leq 1-2\sqrt{v(1-v)}\ (0<v<1)$.
\end{proof}


Nous scindons l'étude de la somme \eqref{eq:Mnup:defbis} selon la parité de $\nu(n)$. Posons ainsi 
\begin{equation}\label{def:Mnu:odd;even}
	\begin{gathered}
		M_{\nu,\iota}(x,p):=\sum_{\substack{n\leq x\\ \nu(n)\equiv 1(\modulo 2)}}\chi_{\nu,p}(n),\quad M_{\nu,\pi}(x,p):=\sum_{\substack{n\leq x\\ \nu(n)\equiv 0(\modulo 2)}}\chi_{\nu,p}(n)\quad(3\leq p\leq x),\\
		\Phi_{\nu,k}(x,y):=\sum_{\substack{n\leq x\\ P^-(n)>y\\ \nu(n)=k}}1\quad(k\geq 1,\,3\leq y\leq x).
	\end{gathered}
\end{equation}
Dans la suite, nous travaillons essentiellement sur la somme $M_{\nu,\iota}(x,p)$, plus commode pour les calculs, et précisons les résultats analogues concernant la somme complémentaire $M_{\nu,\pi}(x,p)$ lorsque cela est nécessaire. \par
Rappelons la définition de $\cD_\varepsilon$ en \eqref{def:domaine:D} et posons, pour $0<\varepsilon<\tfrac12$, 
\begin{equation}\label{def:Kxp;MnupIstar}
	\begin{gathered}
		\JJ(x,p):=\big[\tfrac12\kappa_\varepsilon(\beta_p)\log_2 x-1,\tfrac1{\log 2}\log_2x\big]\cap\R_+^*\quad\big((x,p)\in\cD_\varepsilon\big),\\
		\noalign{\vskip-4mm}\\
		M_{\nu,\iota}^*(x,p)=M_{\nu,\varepsilon,\iota}^*(x,p):=\sum_{k\in \JJ(x,p)}\sum_{\substack{a\leq x/p\\ P^+(a)< p\\\nu(a)=k}}\Phi_{\nu,k}\Big(\frac{x}{ap},p\Big).
	\end{gathered}
\end{equation}


\begin{property}\label{prop:eq:Mp:Mpstar}
	Soit $0<\varepsilon<\tfrac12$. Nous avons l'estimation
	\begin{equation}\label{eq:Mp:Mpstar}
		M_{\nu,\iota}(x,p)=M_{\nu,\iota}^*(x,p)+O\Bigg(\frac{x}{p(\log x)^{\gamma_\nu(\beta_p)+\varepsilon}}\Bigg)\quad\big((x,p)\in\cD_\varepsilon\big).
	\end{equation}
\end{property}
\begin{proof}
	En représentant tout entier naturel $n\geq 2$ sous la forme $n=apb$ où $p=\p(n)$, $P^+(a):=q_{1,\omega}(a)\leq p$, $P^-(b):=q_{\omega(n),\omega}(a)>p$ avec la convention $P^+(1)=1$, $P^-(1)=\infty$ et $\nu(a)=\nu(b)$, nous obtenons
	\[M_{\nu,\iota}(x,p)=\sum_{k\leq(\log x)/\log 4}\sum_{\substack{a\leq x/p\\ P^+(a)< p\\ \nu(a)=k}}\sum_{\substack{b\leq x/ap\\ P^-(b)> p\\ \nu(b)=k}}1=\sum_{k\leq(\log x)/\log 4}\sum_{\substack{a\leq x/p\\ P^+(a)< p\\ \nu(a)=k}}\Phi_{\nu,k}\Big(\frac{x}{ap},p\Big).\]
	Par ailleurs, une application de la majoration \eqref{eq:majo:Mp:smallbignu} fournit
	\begin{align*}
		M_{\nu,\iota}(x,p)&=\sum_{n\in\A_{\nu,\varepsilon}(p)}\chi_{\nu,p}(n)+\sum_{n\in[1,x)\smallsetminus\A_{\nu,\varepsilon}(p)}\chi_{\nu,p}(n)\\
		&=M_{\nu,\iota}^*(x,p)+O\bigg(\frac{x}{p(\log x)^{1-2\sqrt{\beta_p(1-\beta_p)}+\varepsilon}}\bigg),
	\end{align*}
	ce qui fournit l'estimation annoncée puisque $\gamma_\nu(v)\leq 1-2\sqrt{v(1-v)}\ (0<v<1)$.
\end{proof}



\section{Estimations liées aux lois locales de $\nu(n)$}\label{sec:etude;lambda}

Posons, pour $t\in\R$, $3\leq y\leq x$,
\begin{equation}\label{def:u;ryt;h0}
	u=u_y:=\frac{\log x}{\log y},\qquad r_{x,t,y}:=\frac{t-1}{\log u},\qquad h_0(z):=\frac{\e^{-\gamma z}}{\Gamma(z+1)}\quad(z\in\C).
\end{equation}
L'énoncé suivant fournit une estimation de $\Phi_{\nu,k}(x,y)$ lorsque $k\in \JJ(x,p)$.


\begin{sloppypar}
	\begin{thmnodot}\label{th:eq:Phinuk:unif}{\bf \cite[cor. 3.3]{bib:papier1}.}
		Soit $A>1$. Sous les conditions $3\leq y\leqslant\sqrt x$, ${1/A\leq r_{x,k,p}\leq A}$, nous avons uniformément
		\begin{equation}\label{eq:Phinuk:unif}
			\Phi_{\nu,k}(x,y)=\frac{h_0(r_{x,k,p})x(\log u)^{k-1}}{(k-1)!\log x}\Big\{1+O\Big(\frac{1}{\log u}\Big)\Big\}.
		\end{equation}
	\end{thmnodot}
\end{sloppypar}


Nous aurons ensuite besoin d'estimations précises pour la quantité
\begin{equation}\label{def:lambdanu}
	\lambda_\nu(k,y):=\sum_{\substack{P^+(n)\leqslant y\\ \nu(n)=k}}\frac1n\quad(y\geq 3,\, k\geq 1),
\end{equation}
 dans différents domaines de valeurs de $k$. Notons également, pour $n\in\N^*$,
\begin{equation}\label{def:Pn}
	P_n(Y):=\sum_{0\leq j\leq n}\frac{Y^j}{j!},\quad\Phi(v):=\frac{1}{\sqrt{2\pi}}\int_{-\infty}^v\e^{-t^2/2}\d t\quad (v\in\R).
\end{equation}
Nous aurons besoin du lemme technique suivant fournissant des estimations de $P_n(v)$ en fonction du rapport $n/v$.

\begin{lemma}\label{l:eq:sommePartExp}
	Soient $n\in\N$, $v\geq 1$. Notant $\varrho:=n/v$, nous avons les estimations
	\begin{align}
		\label{eq:sommePartExp:small}P_n(v)&=\frac{v^n}{(1-\varrho)n!}\Big\{1+O\Big(\frac{\varrho}{v\{1-\varrho\}^2}\Big)\Big\}\quad(\varrho<1),\\
		\label{eq:sommePartExp:uniform}P_n(v)&=\e^v\bigg\{\Phi\bigg(\frac{n-v}{\sqrt v}\bigg)+O\Big(\frac1{\sqrt v}\Big)\bigg\}\quad(v\geq 1,\, n\geq 0),\\
		\label{eq:sommePartExp:big}P_n(v)&=\e^v\bigg\{1+O\bigg(\frac{\sqrt\varrho\e^{-vQ(\varrho)}}{\{\varrho-1\}\sqrt{v}}\bigg)\bigg\}\quad(\varrho>1).
	\end{align}
\end{lemma}
\begin{proof}
	Lorsque $\varrho<1$, nous pouvons écrire
	\[P_n(v)=\frac1{2\pi i}\oint_{|z|=\varrho}\frac{\e^{zv}\d z}{(1-z)z^{n+1}}.\] 
	Puisque, par ailleurs,
	\[\frac1{1-z}=\frac{1}{1-\varrho}+\frac{z-\varrho}{(1-\varrho)^2}+\frac{(z-\varrho)^2}{(1-z)(1-\varrho)^2}\quad(|z|<1),\]
	nous obtenons
	\begin{equation}\label{eq:Pn:decomposition}
		P_n(v)=\frac{v^{n}}{(1-\varrho)n!}+\frac{1}{2\pi i(1-\varrho)^2}\oint_{|z|=\varrho}\frac{\e^{zv}(z-\varrho)^2\d z}{(1-z)z^{n+1}}.
	\end{equation}
	Or, le module de l'intégrale du membre de droite de \eqref{eq:Pn:decomposition} peut être majoré par
	\begin{equation}\label{eq:majo:Pn:laplacemethod}
		\frac{\varrho^{2-n}}{1-\varrho}\int_{-\pi}^\pi\e^{\varrho v\cos\vartheta}|\e^{i\vartheta}-1|^2\d\vartheta\ll\frac{\e^{n}\varrho^{2-n}}{1-\varrho}\int_{-\infty}^\infty\e^{-n\vartheta^2}\vartheta^2\d\vartheta\ll\frac{v^n\varrho}{n!(1-\varrho)v},
	\end{equation}
	d'après la formule de Stirling. La formule \eqref{eq:sommePartExp:small} s'ensuit en regroupant les estimations \eqref{eq:Pn:decomposition} et \eqref{eq:majo:Pn:laplacemethod}.\par
	Afin d'établir la formule \eqref{eq:sommePartExp:uniform},  posons $y:=(n-v)/\sqrt n$, $z:=(n-v)/\sqrt v$. Lorsque $n\leq\tfrac12 v$, l'estimation souhaitée résulte directement de  \eqref{eq:sommePartExp:small} puisque l'on alors
	\[P_n(v)\ll \frac{v^n}{n!}\ll\frac{\e^v}{\sqrt v},\quad\Phi(z)\ll\frac1{\sqrt v}.\] 
	Dans le cas $\tfrac12 v\leq n\leq 2v$, nous avons $y\leq z$, d'où
	\begin{equation}\label{eq:diff:Phis}
		\Phi(z)-\Phi(y)=\frac1{\sqrt{2\pi}}\int_{y}^{z}\e^{-t^2/2}\d t\ll\e^{-y^2/4}(z-y)\ll\frac{\e^{-y^2/4}(n-v)^2}{v\sqrt n}\ll\frac{y^2\e^{-y^2/4}}{\sqrt v}\ll\frac1{\sqrt v}.
	\end{equation}
	Nous sommes donc en mesure d'appliquer \cite[lemme 2.1]{bib:dekoninck}  sous la forme
	\[P_n(v)=\Phi(y)+O\Big(\frac1{\sqrt v}\Big)=\Phi(z)+O\Big(\frac1{\sqrt v}\Big).\]
	Enfin, lorsque $n\geq 2v$, nous avons $\varrho\geq 2$ et la formule \eqref{eq:sommePartExp:uniform} découle de \eqref{eq:majo:norton:sup} qui fournit
	\[P_n(v)=\e^v\Big\{1+O\Big(\frac1{\sqrt v}\Big)\Big\}, \]
alors que $\Phi(z)=1+O(\e^{-v/2})$.
	\par Enfin, la formule \eqref{eq:sommePartExp:big} est une conséquence directe de \eqref{eq:majo:norton:sup} puisque, lorsque $\varrho>1$,
	\[P_n(v)=\e^v-\sum_{j>\varrho v}\frac{v^j}{j!}=\e^v\bigg\{1+O\bigg(\frac{\sqrt\varrho\e^{-vQ(\varrho)}}{\{1-\varrho\}\sqrt v}\bigg)\bigg\}.\qedhere\]
\end{proof}

Pour $n\geq 1$, $y\geqslant 3$, posons
\begin{equation}\label{def:Rn;R}
Y:=2\log_2y,\quad R=R_n(Y):=\frac{P_{n-1}(Y)}{P_n(Y)}.
\end{equation}
Pour $y\geq 3$, $k\geq 1$, notons $\mu_{k,y}:=\sup(1,k/Y)$ et introduisons
\begin{equation}\label{def:Rv}
	\mathfrak R(k,y):=\frac{1}{\log_2 y+(\log y)^{2Q(\mu_{k,y})}\sqrt{\log_2 y}}\quad (k\geq 1,\,y\geq 3),
\end{equation}
où la fonction $Q$ est définie en \eqref{def:Q}. Rappelons enfin les définitions des fonctions $\HH_\nu$,  $\HH_\Omega^*$ en \eqref{def:Hnu;H0}.\par 
 Nous obtenons en premier lieu une estimation uniforme de $\lambda_\Omega(k,y)$.


\begin{theorem}\label{th:eq:lambdaOmega:unif}
	Nous avons, uniformément,
	\begin{equation}\label{eq:lambdaOmega:unif}
		\lambda_\Omega(k,y)=\frac{1+O(\mathfrak R(k,y))}{2^k}\HH_\Omega^*(2R)P_k(2\log_2 y)\quad(k\geq 1,\,y\geq 3).
	\end{equation}
	où $R$ est définie en \eqref{def:Rn;R}.
\end{theorem}

\begin{proof}
	Le membre de gauche de \eqref{eq:lambdaOmega:unif} est le coefficient de $z^k$ de la série
	\[\sum_{P^+(n)\leqslant y}\frac{z^{\Omega(n)}}{n}=\prod_{q\leq y} \Big(1-\frac{z}{q}\Big)^{-1}\quad(|z|<2).\]
	Une forme forte du théorème des nombres premiers permet de réécrire cette somme sous la forme
	\[\sum_{P^+(n)\leqslant y}\frac{z^{\Omega(n)}}{n}=\HH_\Omega(z)(\log y)^z\Big\{1+O\Big(\frac{1}{\log y}\Big)\Big\}.\]
	La formule de Cauchy implique donc, avec la notation \eqref{def:Hnu;H0},
	\begin{equation}\label{eq:lambdaOmega:cauchy}
		\lambda_\Omega(k,y)=\frac{1+O(1/\log y)}{2\pi i}\oint_{|z|=2R}\frac{\HH_\Omega^*(z)\e^{z\log_2 y}}{(1-z/2)z^{k+1}}\d z.
	\end{equation}
	Afin d'alléger les notations, posons 
	\begin{equation}\label{def:Y;varrho}
		Y:=2\log_2 y\quad (y\geq 3),\qquad \varrho:=\frac kY\quad(y\geq 3,\,k\in\N),
	\end{equation}
	et effectuons le changement de variables $s=\tfrac12z$. Nous obtenons
	\[\lambda_\Omega(k,y)=\frac{1+O(1/\log y)}{2^{k}(2\pi i)}\oint_{|s|=R}\frac{\HH_\Omega^*(2s)\e^{sY}}{(1-s)s^{k+1}}\d s.\]
	Remarquons alors que, pour $k\in\N$,
	\[\frac{1}{2\pi i}\oint_{|s|=R}\frac{\e^{sY}}{(1-s)s^{k+1}}\d s=\sum_{0\leq j\leq k}\frac{Y^j}{j!}=P_k(Y).\]
	Or,
	\[\frac{1}{2\pi i}\oint_{|s|=R}\frac{(s-R)\e^{sY}}{(1-s)s^{k+1}}\d s=P_{k-1}(Y)-RP_k(Y)=0,\]
	de sorte qu'en posant 
	\[\varphi(s):=\int_0^1(1-t){\HH_\Omega^*}''(R+t(s-R))\d t\quad (|s|<3),\quad E(k,y):=\frac{1}{2\pi i}\oint_{|s|=R}\frac{\e^{sY}(s-R)^2\varphi(s)}{(1-s)s^{k+1}}\d s,\]
	nous obtenons
	\begin{equation}\label{eq:lambdaOmega:inter:implicitError}
		\lambda_\Omega(k,y)=\frac{\HH_\Omega^*(2R)P_k(Y)+E(k,y)}{2^k}\Big\{1+O\Big(\frac{1}{\log y}\Big)\Big\}.
	\end{equation}
	Traitons le terme d'erreur $E(k,y)$ comme dans \cite[th. 1]{bib:balazard:thesis}. \par
	Supposons dans un premier temps que $k\leq Y-\sqrt Y$. En particulier, $\varrho\leq 1-1/\sqrt Y$. D'après \eqref{eq:sommePartExp:small}, nous avons
	\begin{equation}\label{eq:R:varrho:smallk:util}
		R=\frac{P_{k-1}(Y)}{P_k(Y)}=1-\frac{Y^k}{P_k(Y)k!}=\varrho+O\Big(\frac \varrho{Y(1-\varrho)}\Big)=\varrho+O\Big(\frac 1{\sqrt{Y}}\Big).
	\end{equation}
	Posons 
	\[\varphi^*(z):=\frac{(z-R)^2\varphi(z)}{1-z}\quad(\Re z<1),\]
	de sorte que 
	\[E(k,y)=\frac1{2\pi i}\oint_{|s|=R}\frac{\e^{sY}\varphi^*(s)}{s^{k+1}}\d s.\]
	La fonction $\varphi$ étant holomorphe sur $D(0,3)$, elle est bornée sur le disque unité. Nous avons donc, pour une constante convenable $A_\varphi$,
	\[|\varphi^*(z)|=|R-z|\Big|1+\frac{R-1}{1-z}\Big||\varphi(z)|\leq A_\varphi|R-z|\bigg(1+\frac{\{1-R\}\sqrt Y}{t}\bigg)\leq 4A_\varphi+O(1)\quad(|z|\leq\varrho),\]
	où la constante implicite est absolue. Puisque $s\mapsto\e^{sY}\varphi^*(s)/s^{k+1}$ n'admet aucun pôle dans le domaine $R\leq |z|\leq\varrho$, nous obtenons en vertu de \cite[lemme 3.1]{bib:papier1},
	\[E(k,y)\ll \frac{Y^k(\varrho-R)^2}{(1-\varrho)k!}\ll\frac{P_k(Y)}{Y},\]
	d'après \cite[(3.4)]{bib:papier1}, \eqref{eq:sommePartExp:small} et \eqref{eq:R:varrho:smallk:util}. Ce terme d'erreur est pleinement acceptable au vu de \eqref{eq:lambdaOmega:unif} puisque $\mu_{k,y}=1$ et  donc $\mathfrak R(k,y)\asymp1/Y$.\par
	Supposons ensuite que $k\geq Y+\sqrt Y$, de sorte que $\varrho\geqslant 1+1/\sqrt{Y}$. Nous déduisons alors de \eqref{eq:sommePartExp:big} et de la formule de Stirling que
	\begin{equation}\label{eq:rapportsommepartexp:bigk}
		1-R=\frac{Y^{k}}{k!P_k(Y)}\ll\frac{Y^{k}\e^{k-Y}}{k^{k+1/2}}\ll\frac{\e^{-YQ(\varrho)}}{\sqrt{Y}}.
	\end{equation}
	Puisque $R<1<\varrho$, nous pouvons appliquer le théorème des résidus afin d'obtenir
	\begin{equation}\label{eq:lambdaOmegaerror:residuthm}
		\int_{|s|=\varrho}\frac{\e^{sY}\varphi^*(s)}{s^{k+1}}\d s=E(k,y)-\e^Y(1-R)^2\varphi(1).
	\end{equation}
	Remarquons que $|\varphi^*(s)|\ll 1$ lorsque $|s|=\varrho$. Le module du membre de gauche de \eqref{eq:lambdaOmegaerror:residuthm} peut donc être majoré par 
	\[\frac{\e^{\varrho Y}}{\varrho^{k}}\int_{-\infty}^\infty\e^{-\varrho Y\vartheta^2/\pi^2}\d\vartheta\ll\frac{\e^{\varrho Y}}{\varrho^k\sqrt Y}\ll\frac{\e^{Y\{1-Q(\varrho)\}}}{\sqrt Y}\ll\frac{P_k(Y)}{\sqrt Y\e^{YQ(\varrho)}},\]
	d'après \eqref{eq:sommePartExp:big}. En regroupant les estimations \eqref{eq:rapportsommepartexp:bigk} et \eqref{eq:lambdaOmegaerror:residuthm}, nous obtenons
	\begin{equation}\label{eq:majo:error:mediumk:inter}
		|E(k,y)|\ll\frac{P_k(Y)}{\sqrt Y\e^{YQ(\varrho)}}\quad(k\geq Y+\sqrt Y).
	\end{equation}
	\par Il reste à étudier le cas $|k-Y|\leq \sqrt Y$. Puisque $R<1$, nous avons $|z-R|<|1-z|\ (|z|=R)$ de sorte qu'avec la majoration triviale $P_k(Y)\ll \e^Y$, nous pouvons écrire
	\begin{equation}\label{eq:majo:lambdaOmegaError:mediumk}
		|E(k,y)|\ll\frac{\e^{RY}}{R^{k-1}}\int_{-\infty}^\infty|\vartheta|\e^{-RY\vartheta^2}\d\vartheta\ll \frac{P_k(Y)\e^{Y(R-1-\varrho\log R)}}{Y}.
	\end{equation}
	La majoration \eqref{eq:rapportsommepartexp:bigk} étant encore valable lorsque $|k-Y|\leq \sqrt Y$, nous obtenons
	\begin{equation}\label{eq:majo:lambdaOmegaError:inter:mediumk}
		\e^{Y(R-1-\varrho\log R)}\ll\e^{Y(R-1)(1-\varrho)}\ll1.
	\end{equation}
	En effet, lorsque $\varrho<1$, c'est immédiat puisque $R<1$. Enfin, lorsque $\varrho>1$, nous avons
	\[(R-1)(1-\varrho)\ll\frac{(\varrho-1)\e^{-YQ(\varrho)}}{\sqrt Y}\ll \frac1Y.\]
	Nous déduisons alors le résultat annoncé des estimations \eqref{eq:lambdaOmega:inter:implicitError}, \eqref{eq:majo:error:mediumk:inter}, \eqref{eq:majo:lambdaOmegaError:mediumk} et \eqref{eq:majo:lambdaOmegaError:inter:mediumk}.
\end{proof}

Le Théorème \ref{th:eq:lambdaOmega:unif} fournit une estimation uniforme de la quantité $\lambda_\Omega(k,y)$ en les variables $y$ et $k$. À l'aide des estimations des sommes partielles de la série exponentielle obtenues au Lemme \ref{l:eq:sommePartExp}, nous pouvons déduire de \eqref{eq:lambdaOmega:unif} des estimations effectives de $\lambda_\Omega(k,y)$ dans des domaines restreints de valeurs  de $k$. C'est l'objet des trois énoncés suivants.


\begin{corollary}\label{cor:eq:lambdaOmega:smallk}
	Soit $A\geq 1$. Nous avons, uniformément pour $1\leq k\leq 2\log_2 y-A\sqrt{\log_2 y}$,
	\begin{equation}\label{eq:lambdaOmega:smallk}
		\lambda_\Omega(k,y)=\HH_\Omega\Big(\frac{k}{\log_2 y}\Big)\frac{(\log_2 y)^k}{k!}\Big\{1+O\Big(\frac{k}{A^2\log_2 y}\Big)\Big\}.
	\end{equation}
\end{corollary}

\begin{proof}
	Conservons les notations \eqref{def:Y;varrho}. En injectant l'estimation \eqref{eq:sommePartExp:small} dans \eqref{eq:lambdaOmega:unif}, nous obtenons, pour les valeurs de $k$ considérées,
	\begin{equation}\label{eq:lambdaOmega:inter:smallk}
		\lambda_\Omega(k,y)=\frac{\HH_\Omega^*(2R)(\log_2 y)^k}{(1-\varrho)k!}\Big\{1+O\Big(\frac{k}{A^2\log_2 y}\Big)\Big\}.
	\end{equation}
Or, l'estimation \eqref{eq:sommePartExp:small} fournit également
	\begin{equation}\label{eq:rapportsommepartexp:cor:smallk}
		R=\frac{P_{k-1}(Y)}{P_k(Y)}=\frac{k(1-k/Y)}{Y(1-\{k-1\}/Y)}\Big\{1+O\Big(\frac{k}{A^2Y}\Big)\Big\}=\varrho\Big\{1+O\Big(\frac{k}{A^2Y}\Big)\Big\}.
	\end{equation}
	Puisque $\HH_\Omega^*(z)=(1-z/2)\HH(z)$ $(\Re z<2)$, le résultat annoncé s'ensuit en regroupant les estimations \eqref{eq:lambdaOmega:inter:smallk} et \eqref{eq:rapportsommepartexp:cor:smallk} 
\end{proof}


Rappelons que $\gh=\HH_\Omega^*(2)$.


\begin{corollary}\label{cor:eq:lambdaOmega:bigk}
	Nous avons, uniformément pour $0<\varepsilon\leq\tfrac12$, $(2+\varepsilon)\log_2 y\leq k\leq (\log y)/\log 2$,
	\begin{equation}\label{eq:lambdaOmega:bigk}
		\lambda_\Omega(k,y)=\frac{\gh(\log y)^2}{2^k}\bigg\{1+O\bigg(\frac{\sqrt k}{\varepsilon(\log y)^{2Q(1+\varepsilon/2)}\log_2 y}\bigg)\bigg\}\qquad(y\geq 3).
	\end{equation}
\end{corollary}

\begin{proof}
	Les paramètres $Y$ et $\varrho$ étant définis par \eqref{def:Y;varrho}, nous avons $\varrho\geqslant 1+\tfrac12\varepsilon$. En substituant \eqref{eq:sommePartExp:big} dans \eqref{eq:lambdaOmega:unif}, nous obtenons
	\begin{equation}\label{eq:lambdaOmega:inter:bigk}
		\lambda_\Omega(k,y)=\frac{\HH_\Omega^*(2R)\e^Y}{2^k}\bigg\{1+O\bigg(\frac{\sqrt\varrho}{\varepsilon(\log y)^{2Q(1+\varepsilon/2)}\sqrt{\log_2 y}}\bigg)\bigg\}.
	\end{equation}
	Une nouvelle application de \eqref{eq:sommePartExp:big} permet d'écrire
	\begin{equation}\label{eq:rapportsommepartexp:cor:bigk}
		R=\frac{P_{k-1}(Y)}{P_k(Y)}=1+O\bigg(\frac{\sqrt\varrho\e^{-Q(\varrho)Y}}{\sqrt Y\{\varrho-1\}}\bigg).
	\end{equation}
	L'estimation \eqref{eq:lambdaOmega:bigk} résulte des deux formules précédentes.
\end{proof}


Le Théorème \ref{th:eq:lambdaOmega:unif} permet de préciser le comportement de $\lambda_\Omega(k,y)$ dans le domaine critique des valeurs de $k$, soit  $k\approx2\log_2 y$. Posons, pour $k\in\N$, $y\geq 3$,
\[\Delta_k(t):=\frac{k-t}{\sqrt t}\quad(t\geq 1),\quad\Delta_{k,y}:=\Delta_k(2\log_2 y).\]


\begin{corollary}\label{cor:eq:lambdaOmega:mediumk}
	Nous avons, uniformément pour $y\geq 3$, $k\geq 2$,
	\begin{equation}\label{eq:lambdaOmega:mediumk}
		\lambda_\Omega(k,y)=\frac{\gh(\log y)^2}{2^k}\bigg\{\Phi(\Delta_{k,y})+O\bigg(\frac{1+|\Delta_{k,y}|}{\sqrt{\log_2 y}}\bigg)\bigg\}.
	\end{equation}
\end{corollary}

\begin{proof}
	Conservons les notations \eqref{def:Y;varrho}. En utilisant \eqref{eq:sommePartExp:uniform} afin d'estimer $P_k(Y)$ dans \eqref{eq:lambdaOmega:unif}, nous obtenons
	\begin{equation}\label{eq:lambdaOmega:kbounded:inter}
		\lambda_\Omega(k,y)=\frac{\{1+O(\mathfrak R(k,y))\}\HH_\Omega^*(2R)\e^Y}{2^k}\bigg\{\Phi(\Delta_k(Y))+O\bigg(\frac1{\sqrt Y}\bigg)\bigg\}.
	\end{equation}
	Nous avons $k=Y+\Delta_{k,y}\sqrt Y$. Soit $K$ une constante positive assez grande. Lorsque $\Delta_{k,y}\leq -K$, nous avons $\varrho<1$. Dans ce cas, la formule \eqref{eq:sommePartExp:small} fournit
	\[R=\frac kY\bigg\{1+O\bigg(\frac1{\sqrt Y}\bigg)\bigg\},\]
	de sorte que
	\begin{equation}\label{eq:R:kbounded}
		\HH_\Omega^*(2R)=\HH_\Omega^*\bigg(2\frac kY\bigg\{1+O\bigg(\frac1{\sqrt Y}\bigg)\bigg\}\bigg)=\gh+O\bigg(\frac{|\Delta_k(Y)|}{\sqrt Y}\bigg).
	\end{equation}
	En regroupant \eqref{eq:lambdaOmega:kbounded:inter} et \eqref{eq:R:kbounded}, nous obtenons bien \eqref{eq:lambdaOmega:mediumk}.\par
	Lorsque $|\Delta_{k,y}|\leq K$ ou $|\Delta_{k,y}|\geq K$, les formules \eqref{eq:sommePartExp:uniform} et \eqref{eq:sommePartExp:big} respectivement fournissent
	\[\HH_\Omega^*(2R)=\gh+O\Big(\frac1{\sqrt Y}\Big),\]
	et la formule \eqref{eq:lambdaOmega:mediumk} est alors obtenue en reportant la dernière estimation dans \eqref{eq:lambdaOmega:kbounded:inter}. Cela  complète la démonstration.
\end{proof}
\noindent{\it Remarque.} Bien que valide uniformément en $y$ et $k$, cette estimation n'est pertinente que pour des valeurs de $k$ vérifiant $|k-2\log_2 y|=o(\log_2 y)$.
\par \medskip
Énonçons enfin une estimation de la quantité $\lambda_\omega(k,y)$ dans un domaine restreint de valeurs de $k$. Posons
\[\gr_{t,y}:=\frac{t}{\log_2 y}\quad(t\in\R,\,y\geq 3).\]


\begin{thmnodot}\label{th:eq:lambda_omega:unif}{\bf \cite[th. 3.4]{bib:papier1}.}
	Soient $0<a<b$. Sous les conditions $y\geq 3$, $k\in\N$, $a\leq \gr_{k,y}\leq b$, nous avons uniformément
	\begin{equation}\label{eq:lambda_omega:unif}
		\lambda_\omega(k,y)=\frac{\HH_{\omega}(\gr_{k,y})(\log_2 y)^{k}}{k!}\Big\{1+O\Big(\frac{1}{\log_2 y}\Big)\Big\}.
	\end{equation}
\end{thmnodot}


Rappelons la définition de $\varepsilon_x$ en \eqref{def:betap;betapstar;varepsilonx;deltap;taup}, celle de $\JJ(x,p)$ en \eqref{def:Kxp;MnupIstar}, et remarquons que, pour \mbox{$0<\varepsilon<\tfrac12$}, nous avons
\[\frac{\kappa_\varepsilon(\beta_p)}{2\beta_p}\leq \gr_{k,p}\leq\frac{1}{\beta_p\log 2},\quad \frac{\kappa_\varepsilon(\beta_p)+O(\varepsilon_x)}{2(1-\beta_p)}\leq r_{x,k,p}\leq \frac{1+O(\varepsilon_x)}{(1-\beta_p)\log 2}\quad((x,p)\in\cD_\varepsilon,\,k\in \JJ(x,p)).\]
Rappelons encore la définition de $h_0$ en \eqref{def:u;ryt;h0}, celle de $M_{\nu,\iota}^*(x,p)$ en \eqref{def:Kxp;MnupIstar} et posons, pour $k\geq 1$, $3\leq p\leq x$, 
\begin{equation}\label{def:epsilonx;snu;MnupIsstar}
	s_\nu(k,p):=\frac{h_0(r_{x,k,p})x(\log u_p)^{k-1}\lambda_\nu(k,p)}{(k-1)!\log x},\quad M_{\nu,\iota}^{**}(x,p):=\frac{x}{p\log x}\sum_{k\in \JJ(x,p)}s_\nu(k,p).
\end{equation}


\begin{property}\label{prop:eq:MnupIstar:MnupIsstar}
	Soit $0<\varepsilon<\tfrac12$. Nous avons l'estimation
	\begin{equation}\label{eq:MnupIstar:MnupIsstar}
		M_{\nu,\iota}^*(x,p)=M_{\nu,\iota}^{**}(x,p)\{1+O(\varepsilon_x)\}\quad((x,p)\in\cD_\varepsilon).
	\end{equation}
\end{property}

\begin{proof}
	Notons d'emblée que, d'après \eqref{eq:Mp:Mpstar}, nous pouvons supposer que tout entier $n=apb$ compté dans $M_{\nu,\iota}^*(x,p)$ vérifie $\nu(a)\leq(\log_2 x)/\log 2$, de sorte que $ap^2\leq p^{(\log_2 x)/\log 2+2}< x$ pour $x$ assez grand. Remarquons alors que, pour $x\geq 3$, $k\in\N$,
	\begin{equation}\label{def:vpa;rpka}
		\frac{\log(x/ap)}{\log p}=u_p\{1+O(\varepsilon_x)\},\quad \frac{k-1}{\log \{\log(x/ap)/\log p\}}=r_{x,k,p}\{1+O(\varepsilon_x)\}\quad\Big(a<\frac x{p^2}\Big).
	\end{equation}	
	Puisque nous avons l'encadrement
	\[\frac{\kappa_\varepsilon(\beta_p)+O(\varepsilon_x)}{2(1-\beta_p)}\leq r_{x,k,p}\leq \frac{1+O(\varepsilon_x)}{(1-\beta_p)\log 2},\]
	nous pouvons appliquer le Théorème \ref{th:eq:Phinuk:unif} avec
	\[A=A_\varepsilon:=\max\bigg(\frac4{w\big(\big\{\varepsilon-2\sqrt{\varepsilon(1-\varepsilon)}\big\}/\e\big)},\frac2{\varepsilon\log 2}\bigg),\]
	puisque $\varepsilon>0$ est fixé. Nous obtenons
	\[\Phi_{\nu,k}\Big(\frac x{ap},p\Big)=\frac{h_0(\{k-1\}/\{\log_2(x/ap)-\log_2 p\})x(\log_2(x/ap)-\log_2 p)^{k-1}}{ap(k-1)!\log(x/ap)}\Big\{1+O\Big(\frac{1}{\log u_p}\Big)\Big\}.\]
	Les estimations \eqref{def:vpa;rpka} permettent alors de simplifier l'expression précédente sous la forme
	\begin{equation}\label{eq:Phinuk:up;rpk:inter}
		\Phi_{\nu,k}\Big(\frac x{ap},p\Big)=\frac{h_0(r_{x,k,p})x(\log u_p)^{k-1}}{ap(k-1)!\log x}\Big\{1+O\Big(\frac{1}{\log u_p}\Big)\Big\}.
	\end{equation}
	Le résultat annoncé est alors obtenu en reportant \eqref{eq:Phinuk:up;rpk:inter} dans la définition \eqref{def:Kxp;MnupIstar} de $M_{\nu,\iota}^*(x,p)$.
\end{proof}


Concernant la quantité complémentaire $M_{\nu,\pi}(x,p)$ définie en \eqref{def:Mnu:odd;even}, nous avons $\nu(a)=\nu(b)-1=k-1$ et, en posant, pour $k\geq 1$, $3\leq p\leq x$,
\begin{equation}\label{def:Mnu:caspair}
	\begin{gathered}
		M_{\nu,\pi}^*(x,p)=M_{\nu,\varepsilon,\pi}^*(x,p):=\sum_{k\in \JJ(x,p)}\sum_{\substack{a\leq x/p\\ P^+(a)<p\\\nu(a)=k-1}}\Phi_{\nu,k}\Big(\frac{x}{ap},p\Big),\\
		s_\nu^+(k,p):=\frac{h_0(r_{x,k,p})x(\log u_p)^{k-1}\lambda_\nu(k-1,p)}{(k-1)!\log x},\quad M_{\nu,\pi}^{**}(x,p):=\frac{x}{p\log x}\sum_{k\in\JJ(x,p)}s_\nu^+(k,p),
	\end{gathered}
\end{equation}
l'estimation \eqref{prop:eq:MnupIstar:MnupIsstar} reste valable sous la forme
\begin{equation}\label{eq:Mnustar;Mnusstar:caspair}
	M_{\nu,\pi}^*(x,p)=M_{\nu,\pi}^{**}(x,p)\{1+O(\varepsilon_x)\}\quad((x,p)\in\cD_\varepsilon).
\end{equation}
\par
Les Corollaires \ref{cor:eq:lambdaOmega:smallk} et \ref{cor:eq:lambdaOmega:bigk} mettent en évidence un changement de phase pour $\lambda_\Omega(k,p)$ au passage par la valeur critique $k\approx 2\log_2 p$. Cependant, nous verrons que la somme $M_{\nu,\iota}^{**}(x,p)$ est dominée par des valeurs du rapport $k/\log_2 p$ proches de $\sqrt{(1-\beta_p)/\beta_p}$ pour les grandes valeurs de $p$ et proches de $(1-\beta_p)/2\beta_p$ pour les faibles valeurs de $p$. En particulier, lorsque $\beta_p\approx\tfrac15$, nous avons $k\approx 2\log_2 p$. En conséquence nous scindons l'étude du comportement de $M_{\nu,\iota}^{**}(x,p)$ selon que $\delta_p=\beta_p-\tfrac15$ est ou non proche de 0. Le paragraphe \ref{sec:outcrit} correspond au second cas, le paragraphe \ref{sec:critZone:in} au premier.


\section{Étude de $M_\omega(x,p)$ et de $M_{\Omega}(x,p)$ hors de la zone critique}\label{sec:outcrit}


\subsection{Préparation technique}

Dans toute la suite, fixons $0<\varepsilon<\tfrac12$. Considérons les intervalles
\begin{equation}\label{def:Kxptauj}
	\begin{aligned}
		\cK_{x,p,\delta,1}&:=[\tfrac12\kappa_\varepsilon(\beta_p)\log_2 x-1,(2-|\delta|)\log_2 p],\\
		\cK_{x,p,\delta,2}&:=[(2+|\delta|)\log_2 p,\tfrac1{\log 2}\log_2 x]
	\end{aligned}
	\qquad (3\leq p\leq x, \sqrt{\varepsilon_x}\leq|\delta|<1),
\end{equation}
et définissons,
\begin{equation}\label{def:gsj}
	\begin{aligned}
		\gs_{1}(k,p)&=\gs_{\nu,1}(k,p):=\frac{h_0(r_{x,k,p})\HH_\nu(\gr_{k,p})(\log_2 p)^k(\log u_p)^{k-1}}{k!(k-1)!},\\
		\gs_{2}(k,p)&:=\frac{\gh h_0(r_{x,k,p})(\log p)^2(\log u_p)^{k-1}}{2^k(k-1)!}&
	\end{aligned}
	(3\leq p\leq x,\,k\geq 1),
\end{equation}

Posons encore, pour $3\leq p\leq x$, $j\in\{1,2\}$,
\begin{equation}\label{def:wpjstar;wpj;alphapjstar;alphapj}
	\begin{gathered}
		w_{p,1}^*:=\sqrt{(\log u_p)\log_2 p},\quad w_{p,2}^*:=\tfrac12\log u_p,\quad w_{p,j}:=\lfloor w_{p,j}^*\rfloor,\\
		\alpha_{p,j}^*:=\gr_{w_{p,j}^*,p}=\begin{cases}
			\sqrt{(1-\beta_p)/\beta_p}\quad&\textnormal{si }j=1,\\
			(1-\beta_p)/2\beta_p&\textnormal{si }j=2,
		\end{cases}\qquad\alpha_{p,j}:=\gr_{w_{p,j},p},
	\end{gathered}
\end{equation}
et remarquons d'emblée que, pour $\varepsilon\leqslant \beta_p\leqslant 1-\varepsilon$, nous avons $w_{p,j}^*\asymp \log_2 x\asymp \log_2 p\ (j=1,2)$. Nous utiliserons implicitement ces estimations dans la suite.
\par Rappelons que la fonction polygamma d'ordre $m\in\N$, notée $\psi^{(m)}$ est définie par
\[\psi^{(m)}(z):=\frac{\d^{m+1}\log \Gamma(z)}{\d z^{m+1}}\quad(z\in\C\smallsetminus(-\N)).\]
Notons $\psi:=\psi^{(0)}$ la fonction digamma, $\delta_{ij}$ le symbole de Kronecker et posons
\[b_{j}=b_{x,p,j}:=\frac{(\log 2)w_{p,j}\varepsilon_x}{1+\delta_{2j}}\quad(3\leq p\leq x,\,j=1,2).\]
Notons que $0<b_j<\tfrac12\log 2\ (j=1,2)$. Pour $3\leq p\leq x$, $\sqrt{\varepsilon_x}<|\delta|<1$, considérons les intervalles
\begin{align*}
	\cI_{\Omega,1}=\cI_{\Omega,\delta,1}=\cI_{\Omega,x,p,\delta,1}&:=\big[\tfrac12\kappa_\varepsilon(\beta_p)\log_2 x-1-w_{p,1},(2-|\delta|)\log_2 p-w_{p,1}\big],\\
	\noalign{\vskip-3mm}\\
	\cI_{\Omega,2}=\cI_{\Omega,\delta,2}=\cI_{\Omega,x,p,\delta,2}&:=\big[(2+|\delta|)\log_2 p-w_{p,2},\tfrac1{\log 2}\log_2 x-w_{p,2}\big],\\
	\noalign{\vskip-3mm}\\
	\cI_{\omega}=\cI_{\omega,x,p}&:=\big[\tfrac12\kappa_\varepsilon(\beta_p)\log_2 x-1-w_{p,1},\tfrac1{\log 2}\log_2 x-w_{p,1}\big],
\end{align*}
et posons, pour $j\in\{1,2\}$, 
\begin{gather*}
	\cP_{\Omega,j}:=\Bigg[-\sqrt{\frac{w_{p,j}\log w_{p,j}}{b_j}},\sqrt{\frac{w_{p,j}\log w_{p,j}}{b_j}}\Bigg],\quad \cE_{\Omega,\delta,j}:=\cI_{\Omega,\delta,j}\smallsetminus \cP_{\Omega,j},\\[5pt]
	\cP_{\omega}:=\cP_{\Omega,1},\quad \cE_{\omega}:=\cI_{\omega}\smallsetminus \cP_{\omega}.
\end{gather*}
\par Afin d'estimer la somme intérieure de $M_{\nu,\iota}^{**}(x,p)$ définie en \eqref{def:epsilonx;snu;MnupIsstar}, nous étendons aux valeurs réelles de $k$ la quantité $\gs_j(k,p)\ (j=1,2)$ définie en \eqref{def:gsj}. Définissions ainsi, en rappelant la définition de $h_0(z)$ en \eqref{def:u;ryt;h0},
\begin{equation}\label{def:gsjstar}
	\begin{aligned}
		\gs_{1}^*(t,p)=\gs_{\nu,1}^*(t,p)&:=\frac{\HH_\nu(\gr_{t,p})\e^{-\gamma r_{x,t,p}}(w_{p,1}^*)^{2t}}{\Gamma(1+ r_{x,t,p})(\log u_p)t\Gamma(t)^2}\\
		\noalign{\vskip-2mm}\\
		\gs_{2}^*(t,p)&:=\frac{\gh\e^{-\gamma r_{x,t,p}}(\log p)^2(w_{p,2}^*)^{t-1}}{2\Gamma(1+ r_{x,t,p})\Gamma(t)}
	\end{aligned}
	\qquad(3\leq p\leq x,\,t>0).
\end{equation}
Remarquons que la quantité $\log \gs_{j}^*(t,p)$ est dominée par le terme 
\[\begin{cases}
	2t\log (w_{p,1}^*/t)+2t+O(\log t)\quad&\textnormal{si }j=1,\\
	t\log(w_{p,2}^*/t)+t+O(\log t)&\textnormal{si }j=2.
\end{cases}\]
Cela laisse augurer que le maximum est atteint lorsque $t$ est proche de $w_{p,j}^*$, donc de $w_{p,j}$. Ces considérations mènent à leur tour à supputer que la somme intérieure de \eqref{def:epsilonx;snu;MnupIsstar} restreinte à l'intervalle $\cI_{\Omega,j}$ (respectivement $\cI_\omega$) est dominée par un intervalle de valeurs de $k$ centré en $w_{p,j}$ (respectivement en $w_{p,1}$). Définissons alors, pour $j\in\{1,2\}$, $\sqrt{\varepsilon_x}<|\delta|<1$, $\sign(\delta)=(-1)^{j+1}$, 
\begin{equation}\label{def:ZOmegatauj;Z_omega;erf;erfc}
	\begin{gathered}
		Z_{\Omega,\delta,j}(x,p):=\sum_{h\in\I_{\Omega,\delta,j}}\frac{\gs_j(w_{p,j}+h,p)}{\gs_j(w_{p,j},p)},\quad Z_\omega(x,p):=\sum_{h\in\I_{\omega}}\frac{\gs_{1}(p,w_{p,1}+h)}{\gs_{1}(w_{p,1},p)}\quad(3\leq p\leq x),\\
	\end{gathered}
\end{equation}
Rappelons enfin la définition de $\delta_p$ en \eqref{def:betap;betapstar;varepsilonx;deltap;taup}. Le résultat suivant fournit une estimation des quantités $Z_{\nu}(x,p)$.


\begin{lemma}\label{l:eq:ZOmegataupj;Z_omega:smalltaup} 
	Soit $j\in\{1,2\}$. Sous les conditions $\sqrt{\varepsilon_x}\leq|\delta_p|<1$, $\sign(\delta_p)=(-1)^{j+1}$, nous avons les estimations
	\begin{equation}\label{eq:ZOmegataupj;Z_omega:smalltaup}
		Z_{\Omega,\delta_p,j}(x,p)=\sqrt{(1+\delta_{2j})\pi w_{p,j}}\{1+O(\varepsilon_x)\},\quad Z_{\omega}(x,p)=\sqrt{\pi w_{p,1}}\{1+O(\varepsilon_x)\}\quad (x\geq 3).
	\end{equation}
\end{lemma}

\begin{proof}
	Définissons
	\begin{equation}\label{def:Hpj}
			H_{p,j}(t):=\log \gs_j^*(t,p)\quad (3\leq p\leq x,\,t>0,\,j=1,2).
	\end{equation}
	Notre premier objectif consiste à expliciter un développement de Taylor pour $H_{p,j}(t)$ autour du point $t=w_{p,j}$. Posons,
	\begin{equation}\label{def:sigmanu}
		\sigma_\nu(z):=\begin{cases}
			\displaystyle\sum_{q}\frac{1-z}{q(q- 1+z)}&\textnormal{si }\nu=\omega\quad(\Re z>-1),\\
			\displaystyle\sum_{q}\frac{z}{q(q-z)}+\kappa\qquad&\textnormal{si }\nu=\Omega\quad(\Re z<2).
		\end{cases}
	\end{equation}
	Évaluons alors les dérivées logarithmiques des fonctions de $t$ apparaissant au membre de droite de~\eqref{def:gsjstar}. \par 
	Tout d'abord, un calcul standard permet d'obtenir
	\begin{equation}\label{eq:derivatives:Hnu}
		\frac{\d}{\d t}\log \HH_\nu(\gr_{t,p})=\frac{\alpha_{p,1}\{\sigma_\nu(\gr_{t,p})-\gamma\}}{w_{p,1}},\qquad\frac{\d^m}{\d t^m}\log \HH_\nu(\gr_{t,p})\ll\varepsilon_x^{m}\quad(m\geq 2).
	\end{equation}
	Par ailleurs, pour $m\in\N^*$,
	\begin{equation}\label{eq:derivatives:logGamma}
		\frac{\d^m}{\d t^m}\log\Gamma(1+ r_{x,t,p})=\frac{\psi^{(m-1)}(1+ r_{x,t,p})}{(\log u_p)^m}\ll\varepsilon_x^{m}.
	\end{equation}
	En outre, la formule du produit de Weierstrass pour $\Gamma(t)$ permet d'écrire
	\begin{equation}\label{eq:polygamma:weierstrass}
		\psi^{(m)}(z)=(-1)^{m+1}m!\sum_{k\geq 0}\frac{1}{(z+k)^{m+1}}\qquad(m>0,\,z\in\C\smallsetminus(-\N)).
	\end{equation}
	Enfin, une application de la formule d'Euler-Maclaurin à l'ordre $0$ implique
	\begin{equation}\label{eq:polygamma:DL}
		\begin{gathered}
			\psi(t)=\log t-\frac{1}{2t}+O\Big(\frac{1}{t^2}\Big),\qquad\psi'(t)=\frac{1}{t}+\frac{1}{2t^2}+O\Big(\frac{1}{t^3}\Big),\\
			\psi''(t)=-\frac{1}{t^2}+O\Big(\frac{1}{t^3}\Big),\qquad\psi'''(t)\ll\frac{1}{t^3}.
		\end{gathered}
	\end{equation}
	Par dérivation, nous déduisons  des estimations \eqref{eq:derivatives:Hnu} à \eqref{eq:polygamma:DL} que, pour $3\leq p\leq x$,
	\begin{equation}\label{eq:derivatives:Hpj}
		H_{p,j}'(w_{p,j})\ll\varepsilon_x,\quad H_{p,j}''(w_{p,j})=-\frac{1+\delta_{1j}+O(\varepsilon_x)}{w_{p,j}},\quad H'''_{p,j}(w_{p,j})\ll\varepsilon_x^2,\quad H_{p,j}^{(4)}(w_{p,j})\ll\varepsilon_x^3.
	\end{equation}
	Nous commençons par traiter le cas $\nu=\Omega$. Un développement de Taylor à l'ordre $4$ fournit, pour $h\in \cP_{\Omega,j}\cap\Z$,
	\begin{equation}\label{eq:Hpj:DL}
		H_{p,j}(w_{p,j}+h)=H_{p,j}(w_{p,j})+H'_{p,j}(w_{p,j})h-\frac{h^2}{\{1+\delta_{2j}\}w_{p,j}}+\tfrac16H'''_{p,j}(w_{p,j})h^3+O\big(h^2\varepsilon_x^2+h^4\varepsilon_x^3\big),
	\end{equation}
	dont nous déduisons l'estimation
	\begin{equation}\label{eq:rapportgsjstar}
		\frac{\gs_j^*(w_{p,j}+h,p)}{\gs_j^*(w_{p,j},p)}=\e^{H'_{p,j}(w_{p,j})h-h^2/\{1+\delta_{2j}\}w_{p,j}+\frac16H'''_{p,j}(w_{p,j})h^3}\big\{1+O\big(h^4\varepsilon_x^3\big)\big\}\quad(h\in \cP_{\Omega,j}),
	\end{equation}
	puisque $h^2\varepsilon_x^2\ll h^4\varepsilon_x^3$, pour $h\in \cP_{\Omega,j}$. L'estimation \eqref{eq:rapportgsjstar} fournit par ailleurs
	\begin{equation}\label{eq:gsj:gsjstar}
		\gs_j(w_{p,j}+h,p)=\gs_j^*(w_{p,j}+h,p)\{1+O(\varepsilon_x)\}.
	\end{equation}
	Les évaluations \eqref{eq:rapportgsjstar} et \eqref{eq:gsj:gsjstar} permettent donc d'écrire
	\begin{equation}\label{eq:rapportgsj}
		\frac{\gs_j(w_{p,j}+h,p)}{\gs_j(w_{p,j},p)}=\e^{H'_{p,j}(w_{p,j})h-h^2/\{1+\delta_{2j}\}w_{p,j}+\frac16H'''_{p,j}(w_{p,j})h^3}\big\{1+O\big(\varepsilon_x+h^4\varepsilon_x^3\big)\big\}\quad(h\in \cP_{\Omega,j}).
	\end{equation}
	Désignons respectivement par $Z_{\Omega,j}^{(P)}(x,p)$ et $Z_{\Omega,j}^{(E)}(x,p)$ les contributions à $Z_{\Omega,\delta_p,j}(x,p)$ des intervalles $\cP_{\Omega,j}$ et $\cE_{\Omega,\delta_p,j}$. 
	\par Par sommation sur $h\in \cP_{\Omega,j}$, nous obtenons l'estimation
	\[Z_{\Omega,j}^{(P)}(x,p)=\sum_{h\in \cP_{\Omega,j}} e^{-h^2/\{1+\delta_{2j}\}w_{p,j}}\Big\{1+H'_{p,j}(w_{p,j})h+\tfrac16H'''_{p,j}(w_{p,j})h^3+O\big(\varepsilon_x+h^6\varepsilon_x^4\big)\Big\},\]
	où nous avons utilisé les estimations \eqref{eq:derivatives:Hpj} sous la forme $H'_{p,j}(w_{p,j})H'''_{p,j}(w_{p,j})h^4\ll h^4\varepsilon_x^3\ll\varepsilon_x+h^6\varepsilon_x^4$. Puisque l'intervalle de sommation est symétrique, la contribution des termes impairs est nulle. Il vient donc
	\[Z_{\Omega,j}^{(P)}(x,p)=\sum_{h\in \cP_{\Omega,j}}\e^{-h^2/\{1+\delta_{2j}\}w_{p,j}}\big\{1+O\big(\varepsilon_x+h^6\varepsilon_x^4\big)\big\}.\]
	La formule d'Euler-Maclaurin appliquée à l'ordre 0 fournit alors
	\begin{equation}\label{eq:ZOmegaj:P:EM}
		\begin{aligned}
			Z_{\Omega,j}^{(P)}(x,p)&=\int_{\cP_{\Omega,j}}\e^{-t^2/\{1+\delta_{2j}\}w_{p,j}}\{1+O(\varepsilon_x+t^6\varepsilon_x^4)\}\d t+\frac{1+O(\varepsilon_x)}{w_{p,j}}+O(\varepsilon_x)\\
			&=\sqrt{\{1+\delta_{2j}\}\pi w_{p,j}}\{1+O(\varepsilon_x)\},
		\end{aligned}
	\end{equation}
	puisque
	\[\varepsilon_x^4\int_{\cP_{\Omega,j}} t^6\e^{-t^2/\{1+\delta_{2j}\}w_{p,j}}\d t\ll\varepsilon_x^4w_{p,j}^{7/2}\ll\sqrt{\varepsilon_x}.\]
	\par Il reste à évaluer la contribution de l'intervalle $\cE_{\Omega,j}=\cE_{\Omega,\delta_p,j}$. D'après la formule de Taylor-Lagrange à l'ordre $2$, il existe, pour tout $h\in \cE_{\Omega,j}$, un nombre réel $c_{h,j}\in \JJ(x,p)$ tel que
	\begin{equation}\label{eq:Hpj:DL:errorDomain}
		H_{p,j}(w_{p,j}+h)=H_{p,j}(w_{p,j})+H'_{p,j}(w_{p,j})h+\tfrac12H''_{p,j}(c_{h,j})h^2\quad\big(3\leq p\leq x,\,h\in \cE_{\Omega,j}\big).
	\end{equation}
	Puisque $c_{h,j}\in \JJ(x,p)$, nous avons $\varepsilon_x\log 2\leqslant 1/c_{h,j}\leqslant 2\varepsilon_x/\kappa_\varepsilon(\beta_p)$. Les estimations \eqref{eq:derivatives:Hnu} à \eqref{eq:derivatives:Hpj} fournissent alors
	\begin{equation}\label{eq:derivatives:Hpj:errorDomain}
		H'_{p,j}(w_{p,j})\ll\varepsilon_x,\quad H_{p,j}''(c_{h,j})=-\frac{2+O(\varepsilon_x)}{\{1+\delta_{2j}\}c_{h,j}}\leq-\frac{2b_j}{w_{p,j}}+O(\varepsilon_x^2)\quad\big(h\in \cE_{\Omega,j}\big).
	\end{equation}
	En regroupant les estimations \eqref{eq:Hpj:DL:errorDomain} et \eqref{eq:derivatives:Hpj:errorDomain}, il vient
	\[H_{p,j}(w_{p,j}+h)-H_{p,j}(w_{p,j})\leq -\frac{b_jh^2}{w_{p,j}}+O(1)\quad(h\in \cE_{\Omega,j}),\]
	soit
	\[\frac{\gs_j(w_{p,j}+h,p)}{\gs_j(w_{p,j},p)}\ll\e^{-b_jh^2/w_{p,j}}\quad(h\in \cE_{\Omega,j}).\]
	Une sommation sur $h\in \cE_{\Omega,j}$ fournit alors
	\begin{equation}\label{eq:majo:ZOmegaj:E}
		\begin{aligned}
			Z_{\Omega,j}^{(E)}(x,p)\ll\int_{\cE_{\Omega,j}}\e^{-b_jt^2/w_{p,j}}\d t\ll\sqrt{\varepsilon_x}.
		\end{aligned}
	\end{equation}
	Ainsi la contribution de l'intervalle $\cE_{\Omega,\delta_p,j}$ à $Z_{\Omega,\delta_p,j}(x,p)$ peut être englobée dans le terme d'erreur de $Z_{\Omega,j}^{(P)}(x,p)$. Cela complète la démonstration dans le cas $\nu=\Omega$.
	\par Lorsque $\nu=\omega$, puisque $\cP_\omega=\cP_{\Omega,1}$, nous avons directement $Z_\omega^{(P)}=Z_{\Omega,1}^{(P)}\{1+O(\varepsilon_x)\}$. Par ailleurs, la majoration \eqref{eq:majo:ZOmegaj:E} reste clairement valable lorsque l'intégration s'effectue sur l'intervalle $\cE_\omega$ puisque $\cE_\omega\subset \R\smallsetminus \cP_{\Omega,1}$.
\end{proof}


\subsection{Estimations de $M_\omega(x,p)$ et de $M_\Omega(x,p)$ hors de la zone critique}
Nous sommes désormais en mesure de démontrer le résultat principal de cette section. Rappelons les définitions \eqref{def:fnu;rho;rhonu} et \eqref{def:betap;betapstar;varepsilonx;deltap;taup}.

\begin{theorem}\label{th:eq:Mp:critZone:out}
	Soit $0<\varepsilon<\tfrac12$. Sous la condition $(x,p)\in\cD_\varepsilon$, nous avons uniformément
	\begin{gather}
		\label{eq:Mp:critZone:out:pomega}M_{\omega}(x,p)=\frac{\{1+O(\varepsilon_x)\}\varrho_\omega(\beta_p)x}{p(\log x)^{1-2\sqrt{\beta_p(1-\beta_p)}}\sqrt{\log_2 x}},\\[10pt]
		\label{eq:Mp:critZone:out:Gomega}M_{\Omega}(x,p)=\begin{cases}
			\displaystyle\frac{\gc x}{p(\log x)^{\{1-3\beta_p\}/2}}\bigg\{1+O\bigg(\varepsilon_x+\frac{\sqrt{\varepsilon_x}}{|\delta_p|(\log p)^{\delta_p^2/4}}\bigg)\bigg\}&\textnormal{si }\delta_p\leq-\sqrt{\varepsilon_x},\\[20pt]
			\displaystyle\frac{\{1+O(\varepsilon_x/\delta_p^2)\}\varrho_\Omega(\beta_p)x}{p(\log x)^{1-2\sqrt{\beta_p(1-\beta_p)}}\sqrt{\log_2 x}}\qquad&\textnormal{si }\delta_p\geq\sqrt{\varepsilon_x}.
		\end{cases}
	\end{gather}
\end{theorem}

\begin{proof} 
	Afin d'alléger les notations, posons $\varepsilon_x^*:=\varepsilon_x+\sqrt{\varepsilon_x}/|\delta_p|(\log p)^{\delta_p^{2}/4}$. D'après les définitions de $s_\nu$ en \eqref{def:epsilonx;snu;MnupIsstar}, des fonctions $\gs_j$ en \eqref{def:gsj} et compte tenu des estimations \eqref{eq:lambdaOmega:smallk}, \eqref{eq:lambdaOmega:bigk}, \eqref{eq:lambda_omega:unif}, nous avons 
	\begin{equation}\label{eq:sj:gsj:critZone:out}
		\begin{aligned}
			s_\Omega(k,p)&=\begin{cases}
				\displaystyle\gs_{\Omega,1}(k,p)\big\{1+O\big(\varepsilon_x/\delta_p^2\big)\big\}\qquad&\textnormal{si }k\in\cK_{x,p,\delta_p,1},\\
				\displaystyle\gs_{2}(k,p)\{1+O(\varepsilon_x^*)\}\qquad&\textnormal{si }k\in\cK_{x,p,\delta_p,2},
			\end{cases}\\
			s_\omega(k,p)&=\gs_{\omega,1}(k,p)\{1+O(\varepsilon_x)\}\quad(k\in \JJ(x,p)).
		\end{aligned}
	\end{equation}
	Commençons par traiter le cas $\nu=\omega$. Rappelons la définition de $M_{\nu,\iota}^{**}(x,p)$ en \eqref{def:epsilonx;snu;MnupIsstar}. Les estimations \eqref{eq:MnupIstar:MnupIsstar} et \eqref{eq:sj:gsj:critZone:out} fournissent
	\begin{equation}\label{eq:M_omegapIstar:M_omegapIsstar:inter}
		M_{\omega,\iota}^*(x,p)=M_{\omega,\iota}^{**}(x,p)\{1+O(\varepsilon_x)\}=\frac{\{1+O(\varepsilon_x)\}x\gs_{\omega,1}(w_{p,1},p)Z_\omega(x,p)}{p\log x}\quad\big((x,p)\in\cD_\varepsilon\big).
	\end{equation}
	Par ailleurs, nous avons
	\begin{equation}\label{eq:rwpp;Gamma:DL}
		 r_{w_{p,1}^*,p}=\frac{w_{p,1}^*-1}{(1-\beta_p)\log_2 x}=\frac{1+O(\varepsilon_x)}{\alpha_{p,1}^*},\quad\Gamma(1+ r_{w_{p,1}^*,p})=\Gamma\Big(1+\frac{1}{\alpha_{p,1}^*}\Big)\{1+O(\varepsilon_x)\}.
	\end{equation}
	D'après la définition de $\gs_{\omega,1}^*(t,p)$ en \eqref{def:gsjstar} et les estimations \eqref{eq:gsj:gsjstar} et \eqref{eq:rwpp;Gamma:DL}, il vient
	\begin{equation}\label{eq:gs_omega:gs_omegastar:critZone:out}
		\gs_{\omega,1}(w_{p,1},p)=\gs_{\omega,1}^*(w_{p,1}^*,p)\{1+O(\varepsilon_x)\}=\frac{\{1+O(\varepsilon_x)\}f_\omega(\alpha_{p,1}^*)(w_{p,1}^*)^{2w_{p,1}^*+1}}{\Gamma(w_{p,1}^*+1)^2\log u_p}.
	\end{equation}
	La formule de Stirling fournissant par ailleurs
	\begin{equation}\label{eq:stirling}
		\Gamma(w_{p,1}^*+1)^2=2\pi(w_{p,1}^*)^{2w_{p,1}^*+1}\e^{-2w_{p,1}^*}\{1+O(\varepsilon_x)\},
	\end{equation}
	nous déduisons de \eqref{eq:gs_omega:gs_omegastar:critZone:out} et \eqref{eq:stirling} que
	\begin{equation}\label{eq:gs_omega:critZone:out}
		\gs_{\omega,1}(w_{p,1},p)=\frac{\{1+O(\varepsilon_x)\}f_\omega(\alpha_{p,1}^*)\e^{2w_{p,1}^*}}{2\pi\log u_p}.
	\end{equation}
	En regroupant les estimations \eqref{eq:M_omegapIstar:M_omegapIsstar:inter} et \eqref{eq:gs_omega:critZone:out}, nous obtenons finalement
	\begin{equation}\label{eq:M_omegapIstar:critZone:out}
		M_{\omega,\iota}^*(x,p)=\frac{\{1+O(\varepsilon_x)\}xf_\omega(\alpha_{p,1}^*)(\log x)^{2\sqrt{\beta_p(1-\beta_p)}}Z_\omega(x,p)}{2\pi p(\log x)\log u_p},
	\end{equation}
	soit, d'après \eqref{eq:MnupIstar:MnupIsstar}, et à l'aide de l'estimation de $Z_\omega(x,p)$ en \eqref{eq:ZOmegataupj;Z_omega:smalltaup},
	\begin{equation}\label{eq:M_omegapI:critZone:out}
		M_{\omega,\iota}(x,p)=\frac{\{1+O(\varepsilon_x)\}x\beta_p^{1/4}f_\omega(\alpha_{p,1}^*)}{2\sqrt\pi(1-\beta_p)^{3/4}p(\log x)^{1-2\sqrt{\beta_p(1-\beta_p)}}\sqrt{\log_2 x}}\quad\big((x,p)\in\cD_\varepsilon\big).
	\end{equation}
	\par Examinons désormais le cas $\nu=\Omega$. Notons d'emblée la relation triviale
	\begin{equation}\label{eq:ineg:logpowers}
		\tfrac12-\tfrac32v\leq 1-2\sqrt{v(1-v)}\leq 1-4v-2v\log\frac{1-v}{4v}\quad(0<v<1),
	\end{equation}
	avec double égalité si et seulement si $v=\tfrac15$, qui nous sera utile dans la suite.
	\par Rappelons la définition des intervalles $\cK_{x,p,\delta_p,j}\ (j=1,2)$ en \eqref{def:Kxptauj} et notons $M_{p}^-(x)$ et $M_{p}^+(x)$ les contributions respectives à la somme intérieure de $M_{\Omega,\iota}^{**}(x,p)$ des intervalles $\cK_{x,p,\delta_p,1}$ et $\JJ(x,p)\smallsetminus\cK_{x,p,\delta_p,1}$. Nous évaluons séparément chacune de ces contributions.\par
	Dans le cas $\delta_p\geq\sqrt{\varepsilon_x}$, remarquons que $w_{p,1}\in\cK_{x,p,\delta_p,1}$ de sorte que
	\begin{equation}\label{eq:Mp-:bigbeta:inter}
		M_{p}^-(x)=\frac{\big\{1+O\big(\varepsilon_x/\delta_p^2\big)\big\}x\gs_{\Omega,1}(w_{p,1},p)Z_{\Omega,1}(x,p)}{p\log x}.
	\end{equation}
	Puisque $Z_{\Omega,1}(x,p)=Z_{\omega}(x,p)\{1+O(\varepsilon_x)\}$, nous obtenons directement
	\begin{equation}\label{eq:Mp-:bigbeta}
		M_p^-(x)=\frac{\big\{1+O\big(\varepsilon_x/\delta_p^2\big)\big\}x\beta_p^{1/4}f_\Omega(\alpha_{p,1}^*)}{2\sqrt\pi(1-\beta_p)^{3/4}p(\log x)^{1-2\sqrt{\beta_p(1-\beta_p)}}\sqrt{\log_2 x}}\quad(\delta_p\geq\sqrt{\varepsilon_x}).
	\end{equation}
	Dans le cas $\delta_p\leq-\sqrt{\varepsilon_x}$, nous avons 
	\begin{equation}\label{eq:majo:util:wp1}
		\Bigg(1-\sqrt{\frac{\log w_{p,1}}{w_{p,1}}}\Bigg)w_{p,1}\geq(2-|\delta_p|)\log_2 p.
	\end{equation}
	En rappelant la définition de $\gs_{\Omega,1}(t,p)$ en \eqref{def:gsj}, nous pouvons ainsi écrire
	\begin{equation}\label{eq:Mp-:smallbeta:inter}
		\begin{aligned}
			M_{p}^-(x)&=\frac{x\{1+O(\varepsilon_x/\delta_p^2)\}}{p(\log x)\log u_p}\sum_{k\leq(2-|\delta_p|)\log_2 p}\frac{\HH_{\Omega}(\gr_{k,p})k\e^{-\gamma r_{x,k,p}}w_{p,1}^{2k}}{\Gamma(1+ r_{x,k,p})k!^2}\\
			&\ll\frac{x\sqrt{\log_2 x}}{p(\log x)\log u_p}\sum_{k\leq(2-|\delta_p|)\log_2 p}\frac{\HH_{\Omega}(\gr_{k,p})(2w_{p,1})^{2k}}{(2k)!}.
		\end{aligned}
	\end{equation}
	Or, d'après \eqref{eq:majo:norton:inf} et \eqref{eq:majo:util:wp1}, nous avons 
	\begin{equation}\label{eq:majo:Mp-:smallbeta:inter:somme}
		\sum_{k\leq(2-|\delta_p|)\log_2 p}\frac{\HH_\Omega(\gr_{k,p})(2w_{p,1})^{2k}}{(2k)!}\ll \sum_{\ell\leq \{1-\sqrt{(\log w_{p,1})/w_{p,1}}\}2w_{p,1}}\frac{(2w_{p,1})^\ell}{\ell!}\ll \frac{\e^{2w_{p,1}}}{w_{p,1}\sqrt{\log w_{p,1}}}.
	\end{equation}
	Nous déduisons ainsi des estimations \eqref{eq:Mp-:bigbeta:inter} et \eqref{eq:majo:Mp-:smallbeta:inter:somme}, la majoration
	\begin{equation}\label{eq:majo:Mp-:smallbeta}
		M_{p}^-(x)\ll \frac{x}{p(\log x)^{1-2\sqrt{\beta_p(1-\beta_p)}}(\log_2 x)^{3/2}}\quad\big(\delta_p\leq-\sqrt{\varepsilon_x}\big).
	\end{equation}
	Ce terme d'erreur est pleinement acceptable au vu de \eqref{eq:ineg:logpowers}.
	\par Évaluons maintenant la contribution $M_{p}^+(x)$. Dans le cas $\delta_p\leq-\sqrt{\varepsilon_x}$, nous avons $w_{p,2}\in \cK_{x,p,\delta_p,2}$ de sorte que, d'après \eqref{eq:sj:gsj:critZone:out},
	\begin{equation}\label{eq:Mp+:smallbeta:inter}
		M_{p}^+(x)=\frac{x}{p\log x}\bigg\{\{1+O(\varepsilon_x^*)\}\gs_2(w_{p,2},p)Z_{\Omega,2}(x,p)+\sum_{(2-|\delta_p|)\log_2 p\leq k\leq (2+|\delta_p|)\log_2 p}s_\Omega(k,p)\bigg\}.
	\end{equation}
	Notons qu'avec la majoration triviale $\Phi(v)\leqslant 1\ (v\in\R)$, nous obtenons, d'après \eqref{eq:lambdaOmega:mediumk},
	\begin{equation}\label{eq:majo:sOmega:unif}
		s_\Omega(k,p)\ll\frac{(\log p)^2}{2^k}\quad(k\in \JJ(x,p)).
	\end{equation}
	Puisque, par ailleurs, 
	\[\Bigg(1-\sqrt{\frac{\log w_{p,2}}{w_{p,2}}}\Bigg)w_{p,2}\geq(2+|\delta_p|)\log_2 p,\]
	nous déduisons de \eqref{eq:majo:norton:inf} et \eqref{eq:majo:sOmega:unif} que le deuxième terme du membre de droite de \eqref{eq:Mp+:smallbeta:inter} peut être majoré par
	\begin{equation}\label{eq:majo:Mp+:smallbeta:inter:somme}
		\frac{x(\log p)^2}{p\log x}\sum_{k\leq\{1-\sqrt{(\log w_{p,2})/w_{p,2}}\}w_{p,2}}\frac{w_{p,2}^{k}}{k!}\ll\frac{x}{p(\log x)^{\{1-3\beta_p\}/2}\log_2 x}.
	\end{equation}
	Par ailleurs, une nouvelle application de \eqref{eq:gsj:gsjstar} fournit, au vu de la définition de $\gs_2^*(t,p)$ en \eqref{def:gsjstar},
	\begin{equation}\label{eq:gs2:gs2star}
		\gs_2(w_{p,2},p)=\gs_2^*(w_{p,2},p)\{1+O(\varepsilon_x)\}=\frac{\{1+O(\varepsilon_x^*)\}\gh\e^{-\gamma/2}(\log p)^2(w_{p,2}^*)^{w_{p,2}^*}}{2\Gamma(3/2)\Gamma(1+w_{p,2}^*)}.
	\end{equation}
	En regroupant les estimations \eqref{eq:Mp+:smallbeta:inter}, \eqref{eq:majo:Mp+:smallbeta:inter:somme} et \eqref{eq:gs2:gs2star}, la formule de Stirling fournit 
	\begin{equation}\label{eq:Mp+:smallbeta}
		M_{p}^+(x)=\frac{\{1+O(\varepsilon_x^*)\}\gh x\e^{-\gamma/2}(\log p)^2\e^{w_{p,2}^*}}{\sqrt\pi p\log x}=\frac{\gc x\{1+O(\varepsilon_x^*)\}}{3p(\log x)^{\{1-3\beta_p\}/2}}\quad(\delta_p\leq-\sqrt{\varepsilon_x}).
	\end{equation}
	Dans le cas $\delta_p\geq\sqrt{\varepsilon_x}$, remarquons que $w_{p,2}<(2-\delta_p)\log_2 p$. Une nouvelle application de \eqref{eq:majo:sOmega:unif} fournit donc
	\begin{equation}\label{eq:majo:Mp+:bigbeta:inter}
		M_{p}^+(x)\ll\frac{x}{p\log x}\sum_{(2-\delta_p)\log_2 p<k\leq\tfrac1{\log 2}\log_2 x}\frac{\e^{-\gamma r_{x,k,p}}(\log p)^2(w_{p,2})^{k-1}}{\Gamma(1+ r_{x,k,p})(k-1)!}.
	\end{equation}
	Posons $\xi_p:=2\beta_p(2-\delta_p)/(1-\beta_p)>1\ \big(\delta_p\geq\sqrt{\varepsilon_x}\big)$ de sorte que $\xi_pw_{p,2}=(2-\delta_p)\log_2 p$. L'inégalité~\eqref{eq:majo:norton:sup} implique alors
	\begin{equation}\label{eq:majo:Mp+:bigbeta:inter:somme}
		\sum_{k> (2-\delta_p)\log_2 p}\frac{\e^{-\gamma r_{x,k,p}}(w_{p,2})^{k-1}}{\Gamma(1+ r_{x,k,p})(k-1)!}\ll \frac{\e^{w_{p,2}\{1-Q(\xi_{p})\}}}{(\xi_p-1)\sqrt{w_{p,2}}}\ll\frac{1}{\delta_p(\log x)^{(2-\delta_p)\beta_p\{\log \xi_{p}-1\}}}.
	\end{equation}
	Posons
	\[g(v):=2\sqrt{v(1-v)}-(4-\delta_p)v+(2-\delta_p)v\log\frac{2v\{2-\delta_p\}}{1-v}\quad (0<v<1).\]
	Un développement de Taylor à l'ordre 2 fournit une constante $a>0$ telle que $g(\tfrac15+\delta_p)=a\delta_p+O(\delta_p^2)$. En regroupant les estimations \eqref{eq:majo:Mp+:bigbeta:inter} et \eqref{eq:majo:Mp+:bigbeta:inter:somme}, nous obtenons donc la majoration
	\begin{equation}\label{eq:majo:Mp+:bigbeta}
		M_{p}^+(x)\ll \frac{x\e^{-a\delta_p(\log_2 x)/2}}{\delta_pp(\log x)^{1-2\sqrt{\beta_p(1-\beta_p)}}}\quad\big(\delta_p\geq\sqrt{\varepsilon_x}\big).
	\end{equation}
	Ce terme d'erreur est également satisfaisant.
	\par Il reste à étudier la somme complémentaire $M_{\nu,\pi}(x,p)$ définie en \eqref{def:Mnu:odd;even}. Rappelons les définitions de $s_\nu^+(k,p)$, $M_{\nu,\pi}^*(x,p)$ et $M_{\nu,\pi}^{**}(x,p)$ en \eqref{def:Mnu:caspair}. Puisque,
	\[s_\nu^+(k,p)=\frac{\lambda_\nu(k-1,p)s_\nu(k,p)}{\lambda_\nu(k,p)},\]
	nous obtenons, d'après \eqref{eq:Mnustar;Mnusstar:caspair},
	\begin{equation}\label{eq:MOmegapPstarMOmegapIstar}
		M_{\Omega,\pi}^{*}(x,p)=\begin{cases}
			2M_{\Omega,\iota}^{*}(x,p)\{1+O(\varepsilon_x)\}\quad&\textnormal{si }\delta_p\leq-\sqrt{\varepsilon_x},\\
			\sqrt{(1-\beta_p)/\beta_p}M_{\Omega,\iota}^{*}(x,p)\{1+O(\varepsilon_x)\}&\textnormal{si }\delta_p\geq\sqrt{\varepsilon_x}.
		\end{cases}
	\end{equation}
	De même, $M_{\omega,\pi}^{*}(x,p)=\sqrt{(1-\beta_p)/\beta_p}M_{\omega,\iota}^{*}(x,p)\{1+O(\varepsilon_x)\}\ ((x,p)\in\cD_\varepsilon)$. Le résultat annoncé s'ensuit en vertu des estimations \eqref{eq:Mp:Mpstar}, \eqref{eq:MnupIstar:MnupIsstar}, \eqref{eq:Mp-:bigbeta}, \eqref{eq:majo:Mp-:smallbeta}, \eqref{eq:Mp+:smallbeta}, \eqref{eq:majo:Mp+:bigbeta} et \eqref{eq:MOmegapPstarMOmegapIstar}.
\end{proof}




\section{Étude de $M_{\Omega}(x,p)$ dans la zone critique}\label{sec:critZone:in}
\subsection{Préparation technique}

Dans la suite, posons, pour $p\geq 3$,
\[w_p^*:=2\log_2 p,\qquad w_p:=\lfloor w_p^*\rfloor,\]
et considérons les intervalles 
\begin{equation}\label{def:intervals:I;P;E}
	\begin{gathered}
		\cI=\cI_{x,p}:=\big[\tfrac12\kappa_\varepsilon(\beta_p)\log_2 x-1-w_p,\tfrac1{\log 2}\log_2 x-w_p\big],\\
		\cP:=\big[-c(4\log_2 p)^{2/3},c(4\log_2 p)^{2/3}\big],\qquad \cE:=\I\smallsetminus P,
	\end{gathered}
\end{equation}
où $c$ est une constante absolue assez grande.
\par Définissons enfin
\[Z(x,p):=\sum_{h\in\I}\frac{s_\Omega(w_p+h,p)}{s_\Omega(w_p,p)}\quad(3\leq p\leq x),\]
et, pour $K>0$, 
\[\sJ_{K}(a,b):=\int_{-Ka^{2/3}}^{Ka^{2/3}}\Phi\bigg(\frac{t}{\sqrt a}\bigg)\e^{bt-t^2/2a}\d t\quad(a>0,\, b\in\R).\]
L'évaluation de $Z(x,p)$ nécessite une estimation précise de $\sJ_{K}(a,b)$ sous certaines conditions portant sur $a$ et $b$.


\begin{lemma}\label{l:eq:sJtauK} Soient $a, b,\,K$ trois nombres réels tels que $a>0$, $|b|\leq \tfrac12Ka^{-1/3}$. Nous avons
	\begin{equation}\label{eq:sJtauK}
		\sJ_{K}(a,b)=\sqrt{2\pi a}\e^{ab^2/2}\Phi\Big(\frac{b\sqrt a}{\sqrt 2}\Big)\big\{1+O\big(\e^{-K^2a^{1/3}/8}\big)\big\},
	\end{equation}
	où la constante implicite est absolue.
\end{lemma}

\begin{proof}
	Par définition de $\Phi$ nous avons
	\[\sJ_K(a,b)=\frac{\e^{ab^2/2}}{\sqrt{2\pi}}\int_{-Ka^{2/3}}^{Ka^{2/3}}\bigg\{\int_{-\infty}^{t/\sqrt a}\e^{-z^2/2-(t/\sqrt a-b\sqrt a)^2/2}\d z\bigg\}\d t.\]
	À l'aide du changement de variables $u=z-t/\sqrt a$ puis d'une interversion d'intégrales, nous obtenons
	\begin{equation}\label{eq:sJ:interversionint}
		\begin{aligned}
			\sJ_K(a,b)&=\frac{\e^{ab^2/2}}{\sqrt{2\pi}}\int_{-\infty}^0\bigg\{\int_{-Ka^{2/3}}^{Ka^{2/3}}\e^{-(t/\sqrt a+u)^2/2-(t/\sqrt a-b\sqrt a)^2/2}\d t\bigg\}\d u\\
			&=\frac{\sqrt a\e^{ab^2/2}}{\sqrt{2}}\int_{-\infty}^0\e^{-(b\sqrt a+u)^2/4}\Big\{1+O\Big(\e^{-(2Ka^{1/6}+\{b\sqrt a-u\})^2/4}\Big)\Big\}\d u
		\end{aligned}
	\end{equation}
	Remarquons tout d'abord qu'en posant $v=\{b\sqrt a+u\}/\sqrt 2$, nous avons
	\begin{equation}\label{eq:sJ:eval:mainterm}
		\int_{-\infty}^0\e^{-(b\sqrt a+u)^2/4}\d u=\sqrt 2\int_{-\infty}^{b\sqrt{a/2}}\e^{-v^2/2}\d v=2\sqrt\pi\Phi\Big(\frac{b\sqrt a}{\sqrt 2}\Big).
	\end{equation}
	Par ailleurs, la majoration 
	\begin{equation}\label{eq:sJ:eval:errorterm}
		\begin{aligned}
			\int_{-\infty}^0\e^{-(u+b\sqrt a)^2/4-(u+2Ka^{1/6}-b\sqrt a)^2/4}\d u&\ll \int_{\R}\e^{-(u+b\sqrt a)^2/4-(u+2Ka^{1/6}-b\sqrt a)^2/4}\d u\\
			&\ll\e^{-(Ka^{1/6}-b\sqrt a)^2/2}\ll\e^{-K^2a^{1/3}/8}
		\end{aligned}
	\end{equation}
	permet de déduire la formule \eqref{eq:sJtauK} des estimations \eqref{eq:sJ:interversionint}, \eqref{eq:sJ:eval:mainterm} et \eqref{eq:sJ:eval:errorterm}.
\end{proof}



Posons
\begin{equation}\label{def:betapstar}
	\beta_p^*:=\log\frac{1-\beta_p}{4\beta_p}\quad(3\leq p\leq x).
\end{equation}
Rappelons la définition de $\delta_p$ en \eqref{def:betap;betapstar;varepsilonx;deltap;taup}.


\begin{lemma}\label{l:eq:Z:bigtaup}
	Nous avons uniformément 
	\begin{equation}\label{eq:Z:bigtaup}
		Z(x,p)=4\sqrt{\pi \log_2 p}(\log p)^{{\beta_p^*}^2}\Phi\big(\beta_p^*\sqrt{\log_2 p}\big)\Big\{1+O\Big(|\delta_p|+\frac{|\delta_p|^3}{\varepsilon_x}\Big)\Big\}\quad \big(3\leq p\leq x,\,|\delta_p|\leq\varepsilon_x^{1/3}\big).
	\end{equation}
\end{lemma}

\begin{proof}
	Rappelons la définition de $s_\Omega(k,p)$ en \eqref{def:epsilonx;snu;MnupIsstar} et de $\gs_2(k,p)$ en \eqref{def:gsj}. D'après l'estimation \eqref{eq:lambdaOmega:mediumk}, nous avons, pour $k\geq 2$,
	\begin{equation}\label{eq:sOmega:Phigs2}
		\begin{aligned}
			s_\Omega(k,p)&=\frac{\gh h_0(r_{x,k,p})\Phi(\Delta_{k,p})(\log p)^2(\log u_p)^{k-1}}{2^k(k-1)!}\bigg\{1+O\bigg(\varepsilon_x+\frac{1+|\Delta_{k,p}|}{\sqrt{\log_2 x}}\bigg)\bigg\}\\
			&=\Phi(\Delta_{k,p})\gs_2(k,p)\bigg\{1+O\bigg(\frac{1+|\Delta_{k,p}|}{\sqrt{\log_2 x}}\bigg)\bigg\}.
		\end{aligned}
	\end{equation}
	Rappelons les définitions des intervalles $\cI$ et $\cP$ en \eqref{def:intervals:I;P;E}. En remarquant que
	\[\frac{|\Delta_{w_p+h,p}|}{\sqrt{\log_2 x}}\ll h\varepsilon_x\quad(h\in \cI),\]
	nous obtenons, d'après \eqref{eq:rapportgsj}, l'estimation
	\begin{equation}\label{eq:rapportsOmega:critZone:in}
		\frac{s_\Omega(w_p+h,p)}{s_\Omega(w_p,p)}=\big\{1+O\big(\sqrt{\varepsilon_x}+h\varepsilon_x+h^3\varepsilon_x^2\big)\big\}\frac{\Phi(\Delta_{w_p+h,p})\e^{H'_{p,2}(w_p)h+H''_{p,2}(w_p)h^2/2}}{\Phi(\Delta_{w_p,p})}\quad (h\in \cP).
	\end{equation}
	Or, nous avons d'une part
	\begin{equation}\label{eq:rapportsPhi:critZone:in}
		\frac{\Phi(\Delta_{w_p+h,p})}{\Phi(\Delta_{w_p,p})}=2\Phi\Big(\frac{h}{\sqrt{w_p}}\Big),
	\end{equation}
	et d'autre part,
	\begin{equation}\label{eq:derivatives:Hp2:critZone:in}
		\begin{gathered}
			H'_{p,2}(w_p)=\log\frac{w_{p,2}}{w_p}+O(\varepsilon_x)=\beta_p^*+O(\varepsilon_x),\\
			\quad H_{p,2}''(w_p)=\frac{-1+O(\varepsilon_x)}{w_p},\quad H_{p,2}^{(m)}(w_p)\ll\varepsilon_x^{m-1}\ (m\geq 3).
		\end{gathered}	
	\end{equation}
	Désignons encore par $Z^{(P)}(x,p)$ et $Z^{(E)}(x,p)$ les contributions à $Z(x,p)$ des intervalles $\cP$ et $\cE$. Nous pouvons ainsi écrire
	\[Z^{(P)}(x,p)=2\sum_{h\in \cP} \Phi\bigg(\frac{h}{\sqrt{w_p}}\bigg)e^{\beta_p^*h-h^2/2w_p}\big\{1+O\big(\sqrt{\varepsilon_x}+h\varepsilon_x+h^3\varepsilon_x^2\big)\big\}.\]
	La formule d'Euler-Maclaurin appliquée à l'ordre $0$ fournit alors
	\[Z^{(P)}(x,p)=2\int_{\cP}\Phi\bigg(\frac{t}{\sqrt{w_p}}\bigg)e^{\beta_p^*t-t^2/2w_p}\Big\{1+O\big(\sqrt{\varepsilon_x}+t\varepsilon_x+t^3\varepsilon_x^2\big)\Big\}\d t+O(1).\]
	Remarquons également que nous avons $\beta_p^*\leq \tfrac12c(w_p)^{-1/3}$. En effet, d'après la définition de $\beta_p^*$ en \eqref{def:betap;betapstar;varepsilonx;deltap;taup}, nous obtenons
	\[\tfrac12c(w_p)^{-1/3}=\tfrac12c\Big(\frac{\varepsilon_x}{2\beta_p}\Big)^{1/3}\geq 7\varepsilon_{x}^{1/3}\geq\beta_p^*.\]
	Nous pouvons ainsi appliquer \eqref{eq:sJtauK} et obtenir
	\begin{align*}
		Z^{(P)}(x,p)&=4\sqrt{\pi \log_2 p}(\log p)^{{\beta_p^*}^2}\Phi\big(\beta_p^*\sqrt{\log_2 p}\big)\bigg\{1+O\bigg(|\delta_p|+\frac{|\delta_p|^3}{\varepsilon_x}\bigg)\bigg\},
	\end{align*}
	puisque, d'une part,
	\[\varepsilon_x\int_{\cP} t\Phi\bigg(\frac{t}{\sqrt{w_p}}\bigg)e^{\beta_p^*t-t^2/2w_p}\d t\ll |\delta_p|\sqrt{\log_2 p}(\log p)^{{\beta_p^*}^2}\Phi\big(\beta_p^*\sqrt{\log_2 p}\big),\]
	et d'autre part
	\[\varepsilon_x^2\int_{\cP} t^3\Phi\bigg(\frac{t}{\sqrt{w_p}}\bigg)e^{\beta_p^*t-t^2/2w_p}\d t\ll \frac{|\delta_p|^3}{\varepsilon_x}\sqrt{\log_2 p}(\log p)^{{\beta_p^*}^2}\Phi\big(\beta_p^*\sqrt{\log_2 p}\big).\]
	\par Il reste à évaluer la contribution $Z^{(E)}(x,p)$ de l'intervalle $\cE$ à $Z(x,p)$. D'après la formule de Taylor-Lagrange à l'ordre $2$, il existe, pour tout $h\in \cE$, un nombre réel $c_{h}\in \JJ(x,p)$ tel que
	\begin{equation}\label{eq:Hp2:DL:errorDomain:critZone:in}
		H_{p,2}(w_p+h)=H_{p,2}(w_p)+H'_{p,2}(w_p)h+\tfrac12H''_{p,2}(c_{h})h^2\quad\big(3\leq p\leq x,\,h\in \cE\big).
	\end{equation}
	Puisque $c_{h}\in \JJ(x,p)$, nous avons $\varepsilon_x\log 2\leqslant 1/c_{h}\leqslant 2\varepsilon_x/\kappa_\varepsilon(\beta_p)$. Posons $b=b_p:=2\beta_p\log 2$. Les estimations \eqref{eq:derivatives:Hnu} à \eqref{eq:derivatives:Hpj} fournissent alors, pour $h\in \cE$,
	\begin{equation}\label{eq:derivatives:Hp2:errorDomain:critZone:in}
		H'_{p,2}(w_p)=\beta_p^*+O(\varepsilon_x),\quad H_{p,2}''(c_{h})=-\frac{1+O(\varepsilon_x)}{c_{h}}\leq-\frac{b}{w_p}+O(\varepsilon_x^2).
	\end{equation}
	En regroupant les estimations \eqref{eq:Hp2:DL:errorDomain:critZone:in} et \eqref{eq:derivatives:Hp2:errorDomain:critZone:in}, il vient
	\[H_{p,2}(w_p+h)-H_{p,2}(w_p)\leq \beta_p^*h-\frac{bh^2}{2w_p}+O(1)\quad(h\in \cE),\]
	soit
	\[\frac{s_\Omega(w_p+h,p)}{s_\Omega(w_p,p)}\ll\e^{\beta_p^*h-bh^2/2w_p}\quad(h\in \cE).\]
	Une sommation sur $h\in \cE$ fournit alors
	\begin{align*}
		Z^{(E)}(x,p)&\ll\int_{\cE}\Phi\Big(\frac{t}{\sqrt{w_p}}\Big)\e^{\beta_p^*t-bt^2/2w_p}\d t\\
		&\ll\sqrt{\log_2 p}\e^{(\log_2 p){\beta_p^*}^2/b}\bigg\{\Phi\bigg(\frac{\beta_p^*\sqrt{2\log_2 p}}{\sqrt{b}}-c\sqrt 2\{4\log_2 p\}^{1/6}\bigg)\bigg\}\\
		&\ll\sqrt{\log_2 p}\exp\Big\{\Big[\Big(\frac{25}{4}\Big)^2-\Big(2^{1/3}cb-\frac{25}{4}\Big)^2\Big]\frac{(\log_2 p)^{1/3}}{b}\Big\}\\
		&\ll\sqrt{\log_2 p}.
	\end{align*}
	Ainsi la contribution de l'intervalle $\cE$ à $Z(x,p)$ peut être englobée dans le terme d'erreur de $Z^{(P)}(x,p)$. Cela complète la démonstration.
\end{proof}


Nous sommes désormais en mesure d'évaluer la quantité $M_{\Omega}(x,p)$ définie en \eqref{eq:Mnup:defbis} pour des valeurs de $\beta_p$ situées dans un intervalle centré autour de la valeur critique $\tfrac15$. Rappelons la définition de $\delta_p$ en \eqref{def:betap;betapstar;varepsilonx;deltap;taup}.

\subsection{Estimation de $M_{\Omega}(x,p)$ dans la zone critique}


\begin{theorem}\label{th:eq:Mp:critZone:in}
	Nous avons uniformément
	\begin{equation}\label{eq:Mp:critZone:in}
		M_{\Omega}(x,p)=\bigg\{1+O\bigg(\sqrt{\varepsilon_x}+\frac{|\delta_p|^3}{\varepsilon_x}\bigg)\bigg\}\frac{\gc x\Phi\big(\beta_p^*\sqrt{\log_2 p}\big)}{p(\log x)^{1-\beta_p\{4+2\beta_p^*+{\beta_p^*}^2\}}}\quad\big(x\geq 3,\,|\delta_p|\leq\varepsilon_x^{1/3}\big).
	\end{equation}	
\end{theorem}

\begin{proof}
	Commençons par noter les estimations
	\begin{equation}\label{eq:rwpp;Gamma:critZone:in}
		 r_{w_p^*,p}=\frac{2\beta_p\{1+O(\varepsilon_x)\}}{1-\beta_p}=\tfrac12+O(|\delta_p|),\qquad\Gamma(1+ r_{w_p^*,p})=\Gamma\big(\tfrac32\big)+O(|\delta_p|).
	\end{equation}
	Rappelons la définition de $s_\Omega(k,p)$ en \eqref{def:gsj}. D'après \eqref{eq:sOmega:Phigs2}, nous avons
	\begin{align*}
		s_\Omega(w_p,p)&=\Phi(\Delta_{w_p,p})\gs_2^*(w_p^*,p)\big\{1+O\big(\sqrt{\varepsilon_x}\big)\big\}=\frac{\big\{1+O\big(\sqrt{\varepsilon_x}+|\delta_p|\big)\big\}\gh\e^{-\gamma/2}(\log p)^2(w_{p,2}^*)^{w_p^*-1}}{4\Gamma(3/2)\Gamma(w_p^*)}\\
		&=\{1+O(\sqrt{\varepsilon_x}+|\delta_p|)\}\frac{\gh\e^{-\gamma/2}(\log p)^2(w_p^*)^{w_p^*}}{2\sqrt\pi\Gamma(w_p^*+1)}\Big(\frac{1-\beta_p}{4\beta_p}\Big)^{w_p^*}.
	\end{align*}
	Posons $E_x:=\sqrt{\varepsilon_x}+|\delta_p|^3/\varepsilon_x\ (3\leq p\leq x)$. À l'aide de l'estimation \eqref{eq:Z:bigtaup}, nous obtenons
	\[s_\Omega(w_p,p)Z(x,p)=\{1+O(E_x)\}\frac{\sqrt 2\gh\e^{-\gamma/2}(\log p)^{{2+\beta_p^*}^2}(w_p^*)^{w_p^*+1/2}\e^{\beta_p^*w_p^*}\Phi\big(\beta_p^*\sqrt{\log_2 p}\big)}{\Gamma(w_p^*+1)}.\]
	Par la formule de Stirling, il suit
	\[s_\Omega(w_p,p)Z(x,p)=\bigg\{\frac{\gh\e^{-\gamma/2}}{\sqrt\pi}+O(E_x)\bigg\}(\log p)^{4+2\beta_p^*+{\beta_p^*}^2}\Phi\big(\beta_p^*\sqrt{\log_2 p}\big)\]
	Enfin, en rappelant les définitions \eqref{def:epsilonx;snu;MnupIsstar} et \eqref{def:ZOmegatauj;Z_omega;erf;erfc}, nous obtenons, d'après \eqref{eq:MnupIstar:MnupIsstar},
	\begin{equation}\label{eq:MOmegapIstar:critZone:in}
		M_{\omega,\iota}^*(x,p)=\frac{\{1+O(\varepsilon_x)\}xs_\Omega(w_p,p)Z(x,p)}{p\log x}=\frac{\{1+O(E_x)\}\gc x\Phi(\beta_p^*\sqrt{(\log_2 x)/5})}{3p(\log x)^{1-\beta_p\{4+2\beta_p^*+{\beta_p^*}^2\}}}.
	\end{equation}\par
 	Il reste à étudier la somme complémentaire $M_{\nu,\pi}(x,p)$ définie en \eqref{def:Mnu:odd;even}. Rappelons les définitions \eqref{def:Mnu:caspair}. Nous avons encore
	\[s_\nu^+(k,p)=\frac{\lambda_\nu(k-1,p)s_\nu(k,p)}{\lambda_\nu(k,p)},\]
	de sorte que
	\begin{equation}\label{eq:MOmegapPstar:MOmegapIstar:critZone:in}
		M_{\Omega,P}^*(x,p)=2M_{\omega,\iota}^*(x,p)\{1+O(\sqrt{\varepsilon_x})\}.
	\end{equation}
	La formule \eqref{eq:Mp:critZone:in} se déduit alors de \eqref{eq:Mp:Mpstar}, \eqref{eq:MOmegapIstar:critZone:in} et \eqref{eq:MOmegapPstar:MOmegapIstar:critZone:in}.
\end{proof}


Le résultat suivant précise le comportement de $M_{\nu}(x,p)$ pour des valeurs de $\beta_p$ vérifiant $|\delta_p|\leq\sqrt{\varepsilon_x}$. Rappelons les définitions de $\gamma_\Omega$ et $\Psi$ en \eqref{def:eq:gR;gammanu;Psi}.

\begin{corollary}\label{cor:eq:Mp:critZone:in:details}
	Sous la condition $|\delta_p|\leq\varepsilon_x^{1/3}$, nous avons uniformément
	\begin{equation}\label{eq:MOmega:zonecrit:in:in}
		M_\Omega(x,p)=\frac{\{1+O(\sqrt{\varepsilon_x}+|\delta_p|^3/\varepsilon_x)\}\gc x\Psi(\sqrt{125}\delta_p/4\sqrt{\varepsilon_x})}{p(\log x)^{\gamma_\Omega(\beta_p)}}.
	\end{equation}
\end{corollary}

\begin{proof}
	Considérons les fonctions
	\begin{align*}
		F^+(v)&:=1-v\Big\{4+2\log\frac{1-v}{4v}+\tfrac12\Big(\log\frac{1-v}{4v}\Big)^2\Big\}-\big\{1-2\sqrt{v(1-v)}\big\}\quad(0<v<1),\\
		F^-(v)&:=1-v\Big\{4+2\log\frac{1-v}{4v}+\Big(\log\frac{1-v}{4v}\Big)^2\Big\}-\tfrac12\{1-3v\}\quad(0<v<1).
	\end{align*}
	Deux développements de Taylor consécutifs à l'ordre $4$ au point $v=\tfrac15$ fournissent les estimations
	\begin{equation}\label{eq:deltapowers}
		\begin{aligned}
		F^+(v)=-\tfrac{3125}{768}\big(v-\tfrac15\big)^3+O\big(\big\{v-\tfrac15\}^{4}\big)\quad(v\ll 1),\\
		F^-(v)=-\tfrac{3125}{192}\big(v-\tfrac15\big)^3+O\big(\big\{v-\tfrac15\}^{4}\big)\quad(v\ll 1),
		\end{aligned}
	\end{equation}
	dont nous déduisons
	\begin{equation}\label{eq:logPowers:critZone:in}
		\begin{aligned}
			(\log x)^{1-\beta_p\{4+2\beta_p^*+{\beta_p^*}^2/2\}}&=(\log x)^{1-2\sqrt{\beta_p(1-\beta_p)}}\{1+O(|\delta_p|^3/\varepsilon_x)\}\quad(|\delta_p|\leq\varepsilon_x^{1/3}),\\
			(\log x)^{1-\beta_p\{4+2\beta_p^*+{\beta_p^*}^2\}}&=(\log x)^{\{1-3\beta_p\}/2}\{1+O(|\delta_p|^3/\varepsilon_x)\}\quad(|\delta_p|\leq\varepsilon_x^{1/3}).
		\end{aligned}
	\end{equation}
	Puisque
	\[\beta_p^*=-\tfrac{25}4\delta_p+O(\delta_p^2),\]
	la formule \eqref{eq:MOmega:zonecrit:in:in} annoncée est alors une conséquence immédiate des estimations \eqref{eq:logPowers:critZone:in} appliquées à la formule \eqref{eq:Mp:critZone:in}.\end{proof}
	

Les formules obtenues aux Théorèmes \ref{th:eq:Mp:critZone:out} et \ref{th:eq:Mp:critZone:in} nous permettent d'établir le résultat principal de cet article.



\section{Preuve du Théorème \ref{th:eq:Mp}}\label{sec:recollement}

Les estimations obtenues aux Théorèmes \ref{th:eq:Mp:critZone:out} et \ref{th:eq:Mp:critZone:in} sont complémentaires et compatibles à la frontière de leurs domaines de validité. En effet, d'après la deuxième estimation \eqref{eq:Mp:critZone:out:Gomega}, nous avons, lorsque $\sqrt{\varepsilon_x}\leq\delta_p\leq\varepsilon_x^{1/3}$,
\begin{equation}\label{eq:MOmegap:collation:+}
	\begin{aligned}
		M_{\Omega}(x,p)&=\frac{\big\{1+O\big(\varepsilon_x/\delta_p^2\big)\big\}x\big\{\sqrt{1/5+\delta_p}+\sqrt{4/5-\delta_p}\big\}f_\Omega\big(2-25\delta_p/4+O(\delta_p^2)\big)}{2\sqrt\pi(1/5+\delta_p)^{1/4}(4/5-\delta_p)^{3/4}p(\log x)^{1-2\sqrt{\beta_p(1-\beta_p)}}\sqrt{\log_2 x}}\\
		&=\frac{\big\{1+O\big(\varepsilon_x/\delta_p^2\big)\big\}3x\e^{-\gamma/2}\HH_\Omega(2-25\delta_p/4)\sqrt{10}}{8\sqrt\pi\Gamma(3/2)p(\log x)^{1-2\sqrt{\beta_p(1-\beta_p)}}\sqrt{\log_2 x}},
	\end{aligned}
\end{equation}
puisque $\delta_p\ll\varepsilon_x/\delta_p^2$. Rappelons par ailleurs la définition de $\HH_\Omega^*(z)$ en \eqref{def:Hnu;H0}, de sorte que
\begin{equation}\label{eq:sHOmega:sH0:collation:+}
	\HH_\Omega\big(2-\tfrac{25}4\delta_p\big)=\frac{8\HH_\Omega^*(2-25\delta_p/4)}{25\delta_p}=\frac{\{1+O(\delta_p)\}8\gh}{25\delta_p}.
\end{equation}
En regroupant les estimations \eqref{eq:MOmegap:collation:+} et \eqref{eq:sHOmega:sH0:collation:+}, nous obtenons
\begin{equation}\label{eq:recollement:bigbeta:inter}
	M_{\Omega}(x,p)=\frac{\big\{1+O\big(\varepsilon_x/\delta_p^2\big)\big\}2\sqrt 2\gc x}{5\sqrt{5\pi} p(\log x)^{1-2\sqrt{\beta_p(1-\beta_p)}}\delta_p\sqrt{\log_2 x}},
\end{equation}
Par ailleurs, en remarquant que
\[\Phi(\beta_p^*\sqrt{\log_2 p})=\frac{\big\{1+O\big(\varepsilon_x/\delta_p^2\big)\big\}(\log p)^{-{\beta_p^*}^2/2}}{|\beta_p^*|\sqrt{2\pi\log_2 p}}=\frac{\big\{1+O\big(\varepsilon_x/\delta_p^2\big)\big\}2\sqrt 2}{5\sqrt{5\pi}(\log x)^{\beta_p{\beta_p^*}^2/2}\delta_p\sqrt{\log_2 x}},\]
la formule \eqref{eq:Mp:critZone:in} fournit, pour $\sqrt{\varepsilon_x}\leq\delta_p\leq\varepsilon_x^{1/3}$,
\[M_\Omega(x,p)=\frac{\big\{1+O\big(\varepsilon_x/\delta_p^2+\delta_p^3/\varepsilon_x\big)\big\}2\sqrt 2\gc x}{5\sqrt{5\pi}p(\log x)^{1-\beta_p\{4+2\beta_p^*+{\beta_p^*}^2/2\}}\delta_p\sqrt{\log_2 x}}.\]
La première formule \eqref{eq:logPowers:critZone:in} permet alors de retrouver l'estimation \eqref{eq:recollement:bigbeta:inter} avec terme d'erreur légèrement moins précis sur le domaine $\varepsilon_x^{2/5}\leq\delta_p\leq\varepsilon_x^{1/3}$. En particulier, les formules \eqref{eq:Mp:critZone:out:Gomega} et \eqref{eq:Mp:critZone:in} coïncident pour $\delta_p=\varepsilon_x^{2/5}$.\par
Enfin, dans le domaine $-\varepsilon_x^{1/3}\leq\delta_p\leq-\sqrt{\varepsilon_x}$, la formule \eqref{eq:Mp:critZone:out:Gomega} coïncide avec la formule \eqref{eq:Mp:critZone:in} avec un terme d'erreur plus précis dès lors que $\sqrt{\varepsilon_x}=o(|\delta_p|)$.\par
Cela complète la démonstration.\qed \bigskip

\noindent{\it Remerciements.} L'auteur tient à remercier chaleureusement le professeur Gérald Tenenbaum pour l'ensemble de ses conseils et remarques avisés ainsi que pour ses relectures attentives durant la réalisation de ce travail.




\end{document}